\documentclass[20pt]{amsart}
\usepackage[latin1]{inputenc}
\usepackage{mathrsfs}
\usepackage{cite}
\usepackage{mathtools}
\usepackage{amsmath, amsthm, amsfonts, amssymb}
\usepackage{mathtools}
\usepackage{enumitem}
\usepackage{tikz}
\usepackage{tikz-cd}
\usepackage[all]{xy}

\usepackage{amsthm}
\theoremstyle{plain}
\newtheorem{thm}{Theorem}[section]

\newtheorem{cor}[thm]{Corollary}
\newtheorem{lemma}[thm]{Lemma}
\newtheorem{rmk}[thm]{Remark}
\newtheorem{prop}[thm]{Proposition}

\theoremstyle{definition}
\newtheorem{example}[thm]{Example}
\newtheorem{defn}[thm]{Definition}
\newtheorem{que}[thm]{Question}

 \usepackage{todonotes}

\newcommand{\CHM}{\mathrm{\CHM}}

\usepackage[
urlcolor=blue,
colorlinks=true,
linkcolor=blue,
citecolor=blue,
]{hyperref}

\begin{document}
\title{Bounds on automorphism groups of surfaces of general type with negative $c_2$ $\,\,\,\,\,$}
\author{Xiaokun Zhong}
	\address{School of Mathematics and Computer Science, Quanzhou Normal University, Quanzhou, 362000, P. R. China.}
	\email{xiaokunzhong@163.com}
	\date{\today}
	\makeatletter \@namedef{subjclassname@2020}{\textup{2020} Mathematics Subject Classification} 
     \makeatother

\begin{abstract}
Let $k$ be an algebraically closed field of characteristic $p>0$, and let $S$ be a minimal smooth projective surface of general type over $k$. If the second Chern class satisfies $c_2(S)<0$, then the order of the automorphism group satisfies
\[
|\operatorname{Aut}_k(S)|<2598\,(K_S^2)^4,
\]
and the order of every abelian subgroup $G\subset \operatorname{Aut}_k(S)$ satisfies
\[
|G|<81\,(K_S^2)^3.
\]
\end{abstract}

\maketitle
\vspace*{6pt}
\tableofcontents  

\section{Introduction}

For a smooth projective curve $C$ of genus $g\ge 2$ over the complex numbers $\Bbb{C}$, Hurwitz gave the classical bound (see \cite{s7})
\[
|\operatorname{Aut}(C)|<84(g-1).
\]
This result initiated the systematic study of the orders of automorphism groups of curves, surfaces, and higher-dimensional varieties (see \cite{s1,s2,s5,s6,s9,s10,s14,s15,s17,s18,s19,s20,s21}). In particular, Xiao proved that if $S$ is a minimal smooth projective surface of general type (see \cite{s17,s18}), then
\[
|\operatorname{Aut}(S)|\le 42^2 K_S^2.
\]
Hacon, M$^C$Kernan, and Xu generalized this to higher dimensions: for an $n$-dimensional smooth projective variety $X$ of general type, there exists a constant $c$ depending only on $n$ such that the birational automorphism group satisfies (see \cite{s5})
\[
|\operatorname{Bir}(X)|\le c\cdot \operatorname{vol}(X,K_X).
\]

However, over an algebraically closed field $k$ of characteristic $p>0$, the linear bound in terms of the genus for curves no longer holds. Stichtenoth (see \cite{s14}) proved that if $C$ is a smooth projective curve of genus $g\ge 2$ over $k$, then
\[
|\operatorname{Aut}_k(C)|\le \frac{224}{3}g^4,
\]
and equality holds if and only if $C$ is the quartic Fermat curve and $p=3$.

For a minimal smooth projective surface $S$ of general type, it is still unknown whether the order of the automorphism group as an abstract group can be controlled by a polynomial function in $K_S^2$. Existing results only provide bounds for certain special subgroups of $\operatorname{Aut}_k(S)$. Ballico (see \cite{s1}) proved that there exists a constant $c$ such that, for every subgroup $M$ of $\operatorname{Aut}_k(S)$ of order is coprime to $p$,
\[
|M|\le c\cdot \log(K_S^2)\cdot (K_S^2)^{45/2}.
\]
Moreover, Cai (see \cite{s2}) obtained that if $G$ is an abelian subgroup of $\operatorname{Aut}_k(S)$ whose order is coprime to $p$, then
\[
|G|\le 624 K_S^2+6708.
\]
Szab\'{o} (see \cite{s15}) investigated the relation between $\operatorname{Aut}_k(S)$ and its $p$-subgroups. He proved that if the order of a Sylow $p$-subgroup of $\operatorname{Aut}_k(S)$ is $t$, then
\[
|\operatorname{Aut}_k(S)|<t^9(25K_S^2)^{160}.
\]
Thus, in positive characteristic, bounding the order of the Sylow $p$-subgroup of $\operatorname{Aut}_k(S)$ is a key step.

In this paper, under the assumption that the surface of general type admits a fibration $f:S\to B$, we consider the group $\operatorname{Aut}_f(S)$ consisting of automorphisms preserving the fibration. We estimate the order of its Sylow $p$-subgroup and obtain the following results.
\begin{thm}[see Corollary 4.2]
Let $k$ be an algebraically closed field of characteristic $p>0$, $S$ a minimal smooth projective surface of general type over $k$, and $B$ a smooth projective curve of genus $b\ge 2$ over $k$. Let $f:S\to B$ be a fibration. Then
\[
|\operatorname{Aut}_f(S)|<15519\,(K_S^2)^4.
\]
\end{thm}

As an application, we further obtain bounds for the automorphism group and its abelian subgroups of surfaces of general type with negative second Chern class:
\begin{thm}[see Theorem 4.3]
Let $k$ be an algebraically closed field of characteristic $p>0$, and let $S$ be a minimal smooth projective surface of general type over $k$ with $c_2(S)<0$. Then
\begin{enumerate}
\item[(i)] $|\operatorname{Aut}_k(S)|<2598\,(K_S^2)^4$.
\item[(ii)] If $G$ is an abelian subgroup of $\operatorname{Aut}_k(S)$, then $|G|<81\,(K_S^2)^3$.
\end{enumerate}
\end{thm}

Based on Theorem 1.2, the following question naturally arises.
\begin{que}
Let $k$ be an algebraically closed field $k$ of characteristic $p>0$, and let $S$ be a minimal smooth projective surface of general type over $k$ with $c_2(S)\geq 0$. Does there exist a positive integer $m$ and a constant $c$ such that
$$
|\mathrm{Aut}_k(S)| \leq c(K_S^2)^m\,?
$$
\end{que}

The paper is organized as follows. Section 2 reviews the method in \cite{s21} and the obstructions to it. Section 3 gives a local equation of the generic fiber of a fibered surface. It then completes the estimate of the order of the Sylow $p$-subgroup of the automorphism group of the generic fiber, viewed as a regular curve. Finally, Section 4 turns to surfaces of general type with negative second Chern class, completing the estimate of the orders of their automorphism groups.

\section{Preliminaries}

In this section, let $k$ be an algebraically closed field of characteristic $p>0$, $S$ a minimal smooth projective surface of general type, and $B$ a smooth projective curve of genus $b\geq 2$ over $k$.

\begin{defn}
A morphism $f\colon S \to B$ is called a \emph{fibered surface} if $f$ is flat and $f_{*}\mathcal{O}_{S}=\mathcal{O}_{B}$.
\end{defn}
The group of fibration-preserving automorphisms of $f$ is
\[
\mathrm{Aut}_{f}(S):=\{ (\sigma,\varphi) \in \mathrm{Aut}_{k}(S) \times \mathrm{Aut}_{k}(B) \mid f \circ \sigma =\varphi \circ f \},
\]
which makes the following diagram commute:
\begin{displaymath}
\xymatrix{
S\ar[r]^{\sigma} \ar[d]_{f} & S \ar[d]^{f}\\
B\ar[r]_{\varphi} & B .
}
\end{displaymath}
We set $\mathrm{Aut}_{B}(S):=\{ \sigma \in \mathrm{Aut}_{k}(S) \mid f \circ \sigma = f\}$.

From $\mathrm{Aut}_k(B)\leq \frac{224}{3}b^4$ and the exact sequence
\[
1 \rightarrow \mathrm{Aut}_{B}(S) \rightarrow \mathrm{Aut}_{f}(S) \rightarrow \mathrm{Aut}_{k}(B),
\]
bounding $|\mathrm{Aut}_{f}(S)|$ reduces to bounding $|\mathrm{Aut}_B(S)|$.

Let $C$ be the generic fiber of $f$ and $K$ the function field of $B$. Via the group isomorphism $\mathrm{Aut}_B(S) \cong \mathrm{Aut}_{K}(C)$, it suffices to study $\mathrm{Aut}_{K}(C)$. A partial estimate of $|\mathrm{Aut}_{K}(C)|$ is as follows.

\begin{lemma}[\cite{s2}, Theorem 2.1; \cite{s21}, Lemma 3.3]
Let $g$ be the arithmetic genus of $C$, $t$ the number of singular points on $\overline{C}:=C \times_{K}\overline{K}$, and $\pi \colon \widetilde{C} \to \overline{C}$ the normalization. Then the following assertions hold.
\begin{itemize}
\item[(i)] If $g(\widetilde{C}) \geq 2$, then $|\mathrm{Aut}_K(C)|\leq \frac{224}{3} g^{4}$.
\item[(ii)] If $g(\widetilde{C})=1$, then $|\mathrm{Aut}_K(C)| \leq 24 (g-1)$.
\item[(iii)] If $g(\widetilde{C})=0$ and $t \geq 2$, then $|\mathrm{Aut}_K(C)| \leq g(g-1)(8g+4)$.
\item[(iv)] If $g(\widetilde{C})=0$ and $t=1$, then $\mathrm{Aut}_K(C) \cong H \rtimes M$, where $H$ is a Sylow $p$-subgroup of $\mathrm{Aut}_K(C)$ (possibly trivial) and $M$ is a cyclic subgroup of order at most $8g+4$.
\end{itemize}
\end{lemma}

\begin{rmk}
According to Lemma 2.2(iv), when $\overline{C}$ is a rational curve with a unique singular point, we must bound the Sylow $p$-subgroup $H$ of $\mathrm{Aut}_K(C)$. Under the assumption $p \geq 5$, it was shown in \cite{s21} that $|H|\leq p(2g-2)$.
\end{rmk}

\noindent\textbf{Review of the approach.} Let us recall the approach of \cite{s21}. First, when $\overline{C}$ is a rational curve with a unique singular point, there exists a positive integer $n$ such that the following conditions hold (see \cite{s13}, Proposition 3, Corollary 5 and 6).

\begin{itemize}
\item[(i)] The diagrams
\begin{displaymath}
\xymatrix{
S_{0} \ar[r]^{h_{0}} \ar[d]_{f_{0}} & S_{1} \ar[r]^{h_{1}} \ar[d]_{f_{1}}& \cdots
\ar[r]^{h_{n-2}}& S_{n-1} \ar[r]^{h_{n-1}} \ar[d]_{f_{n-1}} & S_{n}=S \ar[d]^{f_{n}=f}\\
B_{0} \ar[r]^{F_{0}} & B_{1}\ar[r]^{F_{1}} & \cdots  \ar[r]^{F_{n-2}} & B_{n-1} \ar[r]^{F_{n-1}} & B_{n}=B
}
\end{displaymath}
are commutative. Here $F_{i}\colon B_{i} \to B_{i+1}$ is the relative Frobenius morphism, $S_{i}$ is the normalization of $S_{i+1} \times_{B_{i+1}}B_{i}$, and $h_{i}\colon S_{i} \to S_{i+1}$ is the composition of the normalization $S_{i} \to S_{i+1} \times_{B_{i+1}}B_{i}$ with the projection $S_{i+1} \times_{B_{i+1}}B_{i} \to S_{i+1}$.

\item[(ii)] The generic fiber $C_0$ of $f_{0}$ is a smooth rational curve.

\item[(iii)] The generic fiber $C_1$ of $f_{1}$ is a rational curve with a unique non-smooth point of type $y^{p}=x^{r}$, where $(p,r)=1$ and $r\geq2$.
\end{itemize}

\begin{rmk}
Let $C_i$ be the generic fiber of $f_i$ and $K_i$ the function field of $B_i$. Then $C_i$ is the normalization of $C_{i+1} \times_{K_{i+1}}K_{i}$, $K_i=K_{i+1}^{1/p}$, and
$$
H \subseteq \mathrm{Aut}_{K}(C)= \mathrm{Aut}_{K_{n}}(C_{n}) \subseteq \mathrm{Aut}_{K_{n-1}}(C_{n-1})\subseteq \cdots \subseteq \mathrm{Aut}_{K_{0}}(C_{0}).
$$
\end{rmk}

Based on this remark, we can bound $|H|$ by describing the local equations of $C_1$. Lemma 2.5 below gives a concrete characterization of the local equation of $C_1$ (see Remark 3.2 and Lemma 3.1 for the proof idea).

\begin{lemma}
The local equation of $C_1$ at its unique non-smooth point $Q_1=(x_0,y_0)$ is, up to isomorphism, $y^p=P(x)$, where $P(x)$ has two possibilities.
\begin{itemize}
\item[(i)] If $x_0 \in K_1$,
\begin{equation}
y^p=P(x):=x^{r}+b_{l_0}+\sum _{i=1}^{m} b_{l_i}x^{l_ip},
\end{equation}
with $r\geq 2$, $(r,p)=1$, $b_{l_0} \in K_1\setminus K_1^p$, $0<l_1<\cdots <l_m$, and $b_{l_i} \in K_1\setminus K_1^{*p}$ for $1\leq i\leq m$.
\item[(ii)] If $x_0 \notin K_1$, there exist positive integers $s$ and $t$ such that $x_0^{p^s}\in K_1\setminus K_1^p$ and
\begin{equation}
y^p=P(x):=x(x^{p^s}-x_0^{p^s})^t+\varphi(x),
\end{equation}
where $\varphi(x) \in K_1[x^p]$ and no monomial of $\varphi(x)$ belongs to $K_1^p[x^p]\setminus \{0\}$.
\end{itemize}
\end{lemma}

By analyzing the local equation of $C_1$ in Lemma 2.5 (see Lemma 3.10 for a similar proof idea), we obtain the following lemma.

\begin{lemma}[see \cite{s21}, Lemma 3.4]
Let $H_1$ be the Sylow $p$-group of $\mathrm{Aut}_{K_1}(C_1)$. If $p_a(C_1)\geq 2$ and $|H_1| > p(2p_a(C_1)-2)$, then $p\mid(2p_a(C_1)-2)$ and $C_1$ possesses infinitely many $K_1$-rational points.
\end{lemma}

Lemma 2.6 allows us to estimate $H_1$ (or $H\subseteq H_1$) under the following two assumptions:
\begin{enumerate}
\item[I.] $p\mid(2p_a(C_1)-2)$;
\item[II.] $C_1$ possesses infinitely many $K_1$-rational points.
\end{enumerate}
When $p \neq 2$ and $p_a(C_1)\geq 2$, Assumption I implies that the local equation of $C_1$ must be Equation (1). Indeed, the arithmetic genus given by Equation (2) is $\frac{1}{2}(p-1)p^{s}t$, which does not satisfy $p \mid (2p_a(C_1)-2)$. Hence, $p_a(C_1)=\frac{1}{2}(r-1)(p-1)$, and $p \mid (2p_a(C_1)-2)$ holds if and only if $p \mid (r+1)$.

When $p\geq 5$, we have $p_a(C_1)\geq 2$, so $S_1$ is a surface of general type and hence $\mathrm{Aut}_{K_1}(C_1)$ is finite. Together with Assumption II, this finiteness rules out the case in which $r+1\geq l_m p$ and the integers
\[
r+1,\quad r+1-l_1 p,\quad r+1-l_2 p,\quad \cdots,\quad r+1-l_m p
\]
appearing in Equation (1) are all powers of $p$. Otherwise, there would exist integers $u>u_1>u_2>\cdots >u_m$ (with the convention $u_m=-\infty$ if $r+1=l_m p$) such that $r+1=p^u$, $l_i p=p^{u}-p^{u_i}$,
and the local equation would take the form
$$
y^p=x^{p^u-1}+b_{l_0}+\sum_{i=1}^{m}b_{l_i}x^{p^u-p^{u_i}}.
$$
By Assumption II, this would force $C_1$ to admit infinitely many automorphisms:
for any two $K_1$-rational points $(x_1,y_1)$ and $(x_2,y_2)$ with $x_1\neq x_2$, the map
\[
\begin{cases}
\sigma x=\dfrac{x}{1+cx}, \\[7pt]
\sigma y=\dfrac{dx^{p^{u-1}}+y}{(1+cx)^{p^{u-1}}},
\end{cases}
\]
is an automorphism of $C_1$, where $c:=1/x_1-1/x_2$ and $d:=y_1/x_1^{p^{u-1}}-y_2/x_2^{p^{u-1}}$.
This is a contradiction.

Note that every element of $H_1$ can be written as (see \cite{s3}, Lemma 3.2)
\[
\begin{cases}
\sigma x=\dfrac{x}{1+\theta x}, &\theta \in K_1\\[7pt]
\sigma y=f_0(x)+\dfrac{y}{(1+\theta x)^{\lambda}},&f_0(x) \in K_1(x)
\end{cases}
\]
where $\lambda:=(r+1)/p$ and $f_0(x)$ satisfies
\begin{equation*}(*)\,\,\,\,\,\,\,\,\,\,\,\,\,\,\,\,\,
\begin{aligned}
f_0(x)^p (1+\theta x)^{Mp}=&(\theta x^{\lambda p}-\sum_{i=1}^m b_{l_i}x^{l_i p}-b_{l_0})(1+\theta x)^{Mp-\lambda p}\\
&+\sum_{i=1}^m b_{l_i}x^{l_i p}(1+\theta x)^{Mp-l_i p}+b_{l_0}(1+\theta x)^{Mp}
\end{aligned}\,\,\,\,\,\,\,\,\,\,\,\,\,\,\,\,\,\,\,\,\,\,\,\,\,
\end{equation*}
with $M:=\mathrm{max}(l_m, \lambda)$.
Thus, the coefficient of the highest-degree monomial on the right-hand side of $(*)$ must lie in $K_1^p$, which implies that $H_1=\{\mathrm{id}\}$ when $M=l_m$. Therefore, we only need to consider the case $l_m \leq \lambda$, in which $(*)$ becomes
\begin{equation*}(**)\,\,\,\,\,
\begin{aligned}
f_0(x)^p (1+\theta x)^{\lambda p}=&\theta x^{\lambda p}+\sum_{i=1}^m b_{l_i}x^{l_i p}[(1+\theta x)^{\lambda p-l_i p}-1]
+b_{l_0}[(1+\theta x)^{\lambda p}-1].
\end{aligned}
\end{equation*}
By Assumption II, the integers $\lambda p,\,\, \lambda p-l_1 p,\,\, \lambda p-l_2 p,\, \cdots,\, \lambda p-l_m p$ are not all powers of $p$,; hence the coefficient of the second-highest-degree monomial on the right-hand side of $(**)$ is nonconstant, and its coefficient must lie in $K_1^p$. Using this fact, one can complete the estimate for $|H_1|$.

Therefore, under the condition $p \geq 5$, we have proved that 
$$
|\mathrm{Aut}_f(S)| < 15658\,(K_S^2)^4.
$$

\medskip
\noindent\textbf{Obstacles in the the above approach.} When $p = 2$ or $3$, the approach in \cite{s21} encounters two obstacles: 
\begin{enumerate}
\item[1)] the local equation of $C_1$ may take the form of Equation (2).
\item[2)] when $p_a(C_1) = 1$, $\mathrm{Aut}_{K_1}(C_1)$ is not necessarily finite.
\end{enumerate}

\begin{rmk}
In fact, the difficulty caused by the above obstacles arises only in the case $p_a(C_1)=1$. When $p_a(C_1)\geq 2$, the finiteness of $\mathrm{Aut}_{K_1}(C_1)$ remains valid. By applying the method of \cite{s21} (as in Example 2.8), one similarly proves that in the cases $p=2$ or $3$, the Sylow $p$-subgroup $H$ in Lemma 2.2 (iv) satisfies $|H|\leq p(2p_a(C)-2)$.
\end{rmk}

\begin{example}
Assume that $\mathrm{char}(K_1)=2$, and let the local equation of $C_1$ be
$$
y^2=x(x^2+\beta)^3+\gamma_0+\gamma_1(x^2+\beta),
$$
where $\beta,\gamma_0,\gamma_1 \in K_1 \setminus K_1^2$. Let $H_1$ be the Sylow $p$-group of $\mathrm{Aut}_{K_1}(C_1)$. Every nontrivial element of $H_1$ can be written as (see \cite{s3}, Lemma 3.2)
\[
\begin{cases}
\sigma x=\dfrac{x+\beta \theta}{1+\theta x}, &\theta \in K_1\\[7pt]
\sigma y=f_0(x)+f_1(x) y,&f_0(x),f_1(x) \in K_1(x)
\end{cases}
\]
Thus, 
$$
\begin{aligned}
&\,\,\,\,\,\,\,\,f_0(x)^2+f_1(x)^2[x(x^2+\beta)^3+\gamma_0+\gamma_1(x^2+\beta)]\\
&=\frac{\theta (1+\beta \theta^2)^3(x^2+\beta)^4+(1+\beta \theta^2)^4 x(x^2+\beta)^3}{(1+\theta x)^8}+\gamma_0+\frac{\gamma_1 (1+\beta \theta^2)(x^2+\beta)}{(1+\theta x)^2}.
\end{aligned}
$$
Comparing the coefficients of the term $x(x^2+\beta)^3$, we have $f_1=(1+\beta \theta^2)^2/(1+\theta x)^4$. Substituting this into the above equation yields
$$
\begin{aligned}
f_0(x)^2(1+\theta x)^8=&\theta(1+\beta \theta^2)^3(x^2+\beta)^4+\gamma_0 [(1+\theta^2 x^2)^4+(1+\beta \theta^2)^4]\\
&+\gamma_1(1+\beta \theta^2)(x^2+\beta)[(1+\theta^2 x^2)^3+(1+\beta \theta^2)^3]\\
\end{aligned}
$$
On the right-hand side of the above equation, the exponent 3 in $(1+\theta^2 x^2)^3$ is not a power of the characteristic 2. Hence, the existence of infinitely many $K_1$-rational points on $C_1$ does not cause $C_1$ to admit infinitely many automorphisms. Moreover, the right-hand side of this equation equals
$$
[\theta(1+\beta \theta^2)^3+\gamma_0 \theta^8+ \gamma_1(1+\beta \theta^2)\theta^6]\,x^8+\gamma_1(1+\beta \theta^2)^2 \theta^4 \,x^6+\text{lower degree terms}.
$$

Therefore, the coefficient of $x^6$, namely $\gamma_1(1+\beta\theta^2)^2\theta^4$, lies in $K_1^2$. But $\gamma_1 \in K_1\setminus K_1^2$, this forces $\theta =0$, and hence $\sigma =\mathrm{id}$.
\end{example}

To overcome the above obstacles, a natural idea is to look for a new finiteness condition. Observe that $\mathrm{Aut}_{K_{2}}(C_{2})$ is finite, so we can bound $|H|$ by characterizing the local equations of $C_2$. As the normalization of $C_{2} \times_{K_{2}}K_{0}$ is $C_0\cong \mathbb{P}_{K_0}^1$, the local equation of $C_2$ at its unique non-smooth point $Q_2$ has the following rough form:
\begin{equation}
w^{p^2}=a_nz^n+a_{n-1}z^{n-1}+\cdots +a_1z+a_0,
\end{equation}
where $(w,z)$ are local coordinates at $Q_2$ and $a_i \in K_2$.


\section{Local equations of the generic fiber of a surface fibration}
In this section, let $C_i$ be the generic fiber of $f_i$ in Remark 2.4, $K_i$ the corresponding base field, and $p$ the characteristic of these fields. We denote by $Q_1=(x_0,y_0)$ and $Q_2=(z_0,w_0)$ the unique non-smooth points of $C_1$ and $C_2$, respectively. We say $Q_1$(resp. $Q_2$) is \emph{semi-rational} if $x_0\in K_1$(resp. $z_0 \in K_2$).

\subsection{The local equation of $C_1$ and $C_2$}
\begin{lemma}
The local equation of $C_2$ at $Q_2=(z_0,w_0)$ is, up to isomorphism, $w^{p^2}=\widetilde{P}(z)$, where $\widetilde{P}(z)$ has two possibilities.
\begin{itemize}
\item[(i)] If $Q_2$ is semi-rational, then
$$
\widetilde{P}(z)=z^{r}+\widetilde{a}_{l_0}+\sum _{i=1}^{m} \widetilde{a}_{l_i} z^{l_ip},
$$
with $\widetilde{a}_{l_0} \in K_2\setminus K_2^{p^2}$, $r\in \mathbb{Z}_{>1}$, $(r,p)=1$, $0<l_1<l_2<\cdots<l_m$, and $\widetilde{a}_{l_i}z^{l_i p} \in K_2[z^p]\setminus(K_2^*[z])^{p^2}$ for $1 \leq i\leq m$.
\item[(ii)] If $Q_2$ is not semi-rational, then there exist positive integers $s$ and $t$ such that $\widetilde{\beta}:=z_0^{p^s}\in K_2\setminus K_2^p$ and
$$
\widetilde{P}(z)=z(z^{p^s}-\widetilde{\beta})^t+\widetilde\varphi(z),
$$
where $\widetilde\varphi(z) \in K_2[z^p]$ and no monomial of $\widetilde\varphi(z)$ belongs to $K_2^{p^2}[z^{p^2}]\setminus \{0\}$.
\end{itemize}
\end{lemma}

\begin{proof}
Write $\widetilde P(z)=\sum_{i=0}^n a_i z^i$. Then
\[
\widetilde P'(z)=\sum_{i=1}^n ia_i z^{i-1}.
\]

\text{(i)} When $z_0\in K_2$, after a
translation we may suppose that $z_0=0$. The curve
$C_2\times_{K_2}\overline{K_2}$ has a unique singular point, hence
$\widetilde P'(z)$ has the unique root $z=0$ over $\overline{K_2}$. Therefore,
$ a_1=0$, and if $r>1$ is the smallest integer with $ra_r\neq0$, then
\[
\widetilde P'(z)
=z^{r-1}[na_nz^{n-r}+\cdots+(r+1)a_{r+1}z+ra_r].
\]
The second factor has no root over $\overline{K_2}$, and hence it is a nonzero
constant. Consequently,
\[
ia_i=0\quad (i\neq r).
\]
We may assume
\[
\widetilde P(z)=a_rz^r+\sum_{i=0}^{m}a_{l_i}z^{l_ip},
\]
where $(r,p)=1$ and
$0\leq l_0<\cdots<l_m$.

The point $Q_2=(0,\widetilde{P}(0)^{p^{-2}})$ is non-smooth, hence not $K_2$-rational; thus
$\widetilde P(0)\in K_2\setminus K_2^{p^2}$.
It follows that $l_0=0$ and
\[
\widetilde P(z)=a_rz^r+a_{l_0}+\sum_{i=1}^m a_{l_i}z^{l_ip}.
\]
Since $(r,p)=1$, a change of variables of the form
$z\mapsto a_r^{-u}z$, $w\mapsto a_r^{v}w$ with
$ur+vp^2=1$ normalizes the coefficient of $z^r$ to $1$. Moreover, the monomials in \(\sum_{i=1}^m a_{l_i}z^{l_ip}\) that lie in \(K_2^{p^2}[z^{p^2}]\) can be absorbed into \(w^{p^2}\). Thus, we may assume
\[
w^{p^2}=z^r+\widetilde a_{l_0}
+\sum_{i=1}^m\widetilde a_{l_i}z^{l_ip},
\]
with $\widetilde a_{l_0}:=a_{l_0}/a_r^{vp^2}\in K_2\setminus K_2^{p^2}$, and $\widetilde a_{l_i}z^{l_ip}\notin (K_2^*[z])^{p^2}$ for $i\geq 1$.

\text{(ii)} When $z_0 \notin K_2$,
let $r$ be the largest integer such that $ra_r\neq0$. Since $z_0$ is the
unique root of $\widetilde P'(z)$ in $\overline{K_2}$, we have
\[
\widetilde P'(z)=ra_r(z-z_0)^{r-1}.
\]
Write $r-1=p^{s_0} t_0$ with $(p,t_0)=1$.
Then 
$$
\widetilde{P}'(z)=ra_r(z^{p^{s_0}}-z_0^{p^{s_0}})^{t_0}=ra_r(z^{p^{s_0} t_0}-t_0 z_0^{p^{s_0}} z^{p^{s_0}(t_0-1)} +\cdots)\in K_2[z],
$$
which implies $z_0^{p^{s_0}}\in K_2$ and $s_0\geq 1$. Let $s$ be the smallest positive integer with $z_0^{p^s}\in K_2$, and set $t=(r-1)/p^s$. We have $\widetilde{P}'(z)=a_r(z-z_0)^{p^{s}t}$ and 
$$
\widetilde{P}(z)=a_r z(z-z_0)^{p^{s}t}+\varphi(z),
$$
where $\varphi(z) \in K_2[z^p]$.

Again, since $(r,p)=1$, the coefficient of $z(z-z_0)^{p^{s}t}$ can be normalized to $1$, we may assume
\[
w^{p^2}=\widetilde{P}(z)=z(z^{p^s}-\widetilde{\beta})^t+\varphi(z),
\]
with $\widetilde{\beta}=z_0^{p^s}\in K_2\setminus K_2^p$, $\varphi(z) \in K_2[z^p]$, and no monomial of $\varphi(z)$ belongs to $K_2^{p^2}[z^{p^2}]\setminus \{0\}$.
\end{proof}

\begin{rmk}
Lemma 2.5 can be proved using the same method as in the proof of Lemma 3.1.
\end{rmk}

By Lemma 2.5, $p_a(C_1)=\frac{1}{2}(p-1)(r-1)$ or $\frac{1}{2}(p-1)p^s t$. Thus, when $p_a(C_1)=1$ and $p\in \{2,3\}$, applying Lemma 2.5 again yields the following result.

\begin{prop}
Suppose that $p_a(C_1)=1$ and $p\in \{2,3\}$.
\begin{itemize}
\item[(i)] If $p=3$, then $Q_1$ is semi-rational and the local equation of $C_1$ at $Q_1$ is, up to isomorphism,
$$
y^3=x^{2}+a_{l_0}+\sum _{i=1}^{m} a_{l_i}x^{3l_i},
$$
where $0<l_1<l_2<\cdots<l_m$, $a_{l_0}\in K_1\setminus K_1^3$, and $a_{l_i} \in K_1\setminus K_1^{*3}$ for $1\leq i \leq m$.
\item[(ii)] If $p=2$ and $Q_1$ is semi-rational, then the local equation of $C_1$ at $Q_1$ is, up to isomorphism,
$$
y^2=x^{3}+a_{l_0}+\sum _{i=1}^{m} a_{l_i}x^{2l_i},
$$
with $0<l_1<l_2<\cdots<l_m$, $a_{l_0}\in K_1\setminus K_1^2$, and $a_{l_i} \in K_1\setminus K_1^{*2}$ for $1\leq i \leq m$.
\item[(iii)] If $p=2$ and $Q_1$ is not semi-rational, then the local equation of $C_1$ at $Q_1$ is
$$
y^2=x(x^2+\beta)+\varphi(x),
$$
where $\beta:=x_0^2\in K_1\setminus K_1^2$, $\varphi(x) \in K_1[x^2]$, and no monomial of $\varphi(x)$ belongs to $K_1^2[x^2]\setminus \{0\}$.
\end{itemize}
\end{prop}

\subsection{The Sylow $p$-group of $\mathrm{Aut}_{K_2}(C_2)$}
\subsubsection{The local equation of $C_2$ when $Q_2$ is semi-rational and $p_a(C_1)=1$}
\begin{prop}
Suppose that $p_a(C_1)=1$ and $p\in \{2,3\}$. If $Q_2$ is semi-rational, then the local equation of $C_2$ at $Q_2$ is, up to isomorphism,
\begin{equation}
w^{p^2}=z^{5-p}+\widetilde{b}_{0}+\sum _{i=1}^{n} \widetilde{b}_{l_i}z^{pl_i}+\widetilde\psi^p(z),
\end{equation}
where $\widetilde{b}_{0}\in K_2 \setminus K_2^p$, $\widetilde{b}_{l_i} \in K_2\setminus K_2^{*p}$, $0<l_1<\cdots<l_n$, $\widetilde\psi(z)\in K_2[z]$, and no monomial of $\widetilde\psi(z)$ belongs to $K_2^p[z^p]$.
\end{prop}

\begin{proof}
By Lemma 3.1, we may assume that $Q_2=(0,\widetilde{a}_{l_0}^{p^{-2}})$, and that the local equation of $C_2$ at $Q_2$ is
\[
w^{p^2}=z^{r}+\widetilde{a}_{l_0}+\sum_{i=1}^{m}\widetilde{a}_{l_i}z^{pl_i},
\]
where $\widetilde{a}_{l_0}\in K_2\setminus K_2^{p^2}$, $r\in \Bbb{Z}_{>1}$, $(r,p)=1$, $0<l_1<\cdots<l_m$, and $\widetilde{a}_{l_i}z^{pl_i} \in K_2[z^p] \setminus (K_2^{*}[z])^{p^2}$.
Set
\[
\phi(z)=\widetilde{a}_{l_m}^{1/p}z^{l_m}+\dots+\widetilde{a}_{l_1}^{1/p}z^{l_1}
        +\widetilde{a}_{l_0}^{1/p}=:\phi_1(z)+\phi_2(z),
\]
where every monomial of $\phi_1(z)$ lies in $K_1^p[z^p]$ and no monomial of
$\phi_2(z)$ belongs to $K_1^p[z^p]$.

\medskip
\noindent\textbf{Case $p=3$.}
Write $r=3r_1+\delta$ with $\delta\in\{1,2\}$ and $r_1\ge 0$.  Define
\[
x=\begin{cases}
\dfrac{w^3-\phi(z)}{z^{r_1}}, & \delta=1,\\[8pt]
\dfrac{z^{r_1+1}}{w^3-\phi(z)}, & \delta=2,
\end{cases}
\]
so that in either case $z=x^3$ and
\[
w^3 = x^r + \widetilde{a}_{l_0}^{1/3} + \sum_{i=1}^{m} \widetilde{a}_{l_i}^{1/3} x^{3l_i}.
\]

\emph{Subcase (a): $\widetilde{a}_{l_0}\in K_2^3$.}
Then $\widetilde{a}_{l_0}^{1/3}\in K_1^3$ and $\phi_1(z)\neq0$.
Put $w_1 = w - \phi_1(x^3)^{1/3}$, we obtain that
$$
\begin{cases}
w_1^3=x^r, & \text{if $\phi_2(z)=0$},\\[8pt]
w_1^3=x^r+\phi_2(x^3):=x^r +\sum_{j=1}^{n} \widetilde{c}_{l_{i_j}} x^{3l_{i_j}}, &  \text{if $\phi_2(z)\neq 0$},
\end{cases}
$$
where $\widetilde{c}_{l_{i_j}}:= \widetilde{a}_{l_{i_j}}^{1/3}\in K_1\setminus K_1^3$ and $0<l_1\leq  l_{i_1} < \dots < l_{i_n}\le l_m$.

If $\phi_2(z)=0$ or $r < 3l_{i_1}$, then we may set $y = w_1/x^{r_1}$. When $\delta=1$, this gives $y^3 = x + \cdots$, smooth at $x=0$;
when $\delta=2$, this gives $y^3 = x^2 + \cdots$, whose normalization remains smooth.
Both contradict the non-smoothness of $C_1$.
Therefore, $\phi_2(z)\neq 0$ and $r > 3l_{i_1}$.
Now set $y = w_1/x^{l_{i_1}}$ to obtain
\[
y^3 = x^{r-3l_{i_1}} + \widetilde{c}_{l_{i_1}} + \sum_{j=2}^{n} \widetilde{c}_{l_{i_j}} x^{3l_{i_j}-3l_{i_1}},
\]
which is the local equation of $C_1$ because $\widetilde{c}_{l_{i_1}} \in K_1\setminus K_1^3$. Since $p_a(C_1) = r-3l_{i_1}-1=1$, we have $r = 3l_{i_1}+2$.
Consequently $\delta=2$. We may assume the local equation of  $C_2$ is
$$
w^{9} = z^{3l_{i_1}+2}
        + \sum_{j=1}^{n} \widetilde{a}_{l_{i_j}}z^{3l_{i_j}} + \widetilde\psi^3(z),
$$
with $\widetilde{a}_{l_{i_j}}\in K_2\setminus K_2^3$ and $\widetilde\psi(z):=\phi_1(z)\in K_2[z]$. As $C_2$ is regular, we have $l_{i_1}=0$ by Lemma 3.5, which contradicts that $l_{i_1}>0$. Therefore, Subcase (a) is impossible.

\emph{Subcase (b): $\widetilde{a}_{l_0}\notin K_2^3$.}
Now $\widetilde{a}_{l_0}^{1/3}\in K_1\setminus K_1^3$.
Put $y=w-\phi_1(x^3)^{1/3}$, we get
\[
y^3 = x^r + \widetilde{a}_{l_0}^{1/3} +\phi_2(x^3),
\]
which is the local equation of $C_1$ because $\widetilde{a}_{l_0}^{1/3}\notin K_1^3$.
Hence $p_a(C_1) = r-1$.  From $p_a(C_1)=1$ we obtain $r=2$,
i.e.\ $\delta=2$, $r_1=0$, and the equation of $C_2$ simplifies to
\[
w^{9} = z^{2} + \widetilde{a}_{l_0} + \sum_{j=1}^{n} \widetilde{a}_{l_{i_j}}z^{3l_{i_j}} + \widetilde\psi^3(z),
\]
where $\widetilde{a}_{l_0} \in K_2 \setminus K_2^3$, $\phi_2^3(z):=\sum_{j=1}^{n} \widetilde{a}_{l_{i_j}}z^{3l_{i_j}}$, $\widetilde\psi(z):=\phi_1(z) \in K_2[z]$, and no monomial of $\widetilde\psi(z)$ belongs to $K_2^3[z^3]$.

\medskip
\noindent\textbf{Case $p=2$.}
Now $r$ is odd, write $r=2r_1+1$ with $r_1\ge0$.  Set
$$
x=\frac{w^2-\phi(z)}{z^{r_1}},
$$
then we have $z=x^2$ and
\[
w^2 = x^r + \widetilde{a}_{l_0}^{1/2} + \sum_{i=1}^{m} \widetilde{a}_{l_i}^{1/2} x^{2l_i}.
\]

\emph{Subcase (a): $\widetilde{a}_{l_0}\in K_2^2$.}
Set $w_1 = w - \phi_1(x^2)^{1/2}$ to obtain
$$
\begin{cases}
w_1^2=x^r, & \text{if $\phi_2(z)=0$},\\[8pt]
w_1^3=x^r+\phi_2(x^2):=x^r +\sum_{j=1}^{n} \widetilde{c}_{l_{i_j}} x^{3l_{i_j}}, &  \text{if $\phi_2(z)\neq 0$},
\end{cases}
$$
where $\widetilde{c}_{l_{i_j}}:= \widetilde{a}_{l_{i_j}}^{1/3}\in K_1\setminus K_1^3$ and $0 <l_1\leq l_{i_1} < \dots < l_{i_n}\le l_m$.
Similar to Case $p=3$, we have $\phi_2(z) \neq 0$ and $r > 2l_{i_1}$.
Put $y = w_1/x^{l_{i_1}}$ to get
$$
y^2 = x^{r-2l_{i_1}} + \widetilde{c}_{l_{i_1}} + \sum_{j=2}^{n} \widetilde{c}_{l_{i_j}} x^{2l_{i_j}-2l_{i_1}},
$$
which is the local equation of $C_1$.  Then $p_a(C_1) = \frac{1}{2} (r-2l_{i_1}-1 )= 1$, so $r = 2l_{i_1}+3$.
We may assume the local equation of $C_2$ is
\[
w^{4} = z^{2l_{i_1}+3} 
        + \sum_{j=1}^{n} \widetilde{a}_{l_{i_j}}z^{2l_{i_j}} +\widetilde\psi^2(z),
\]
with $\widetilde{a}_{l_{i_j}}\in K_2\setminus K_2^2$ and $\widetilde\psi(z)\in K_2[z]$. Similar to Case $p=3$, as $C_2$ is regular and $l_{i_1}>0$, Subcase (a) is impossible by Lemma 3.5.

\emph{Subcase (b): $\widetilde{a}_{l_0}\notin K_2^2$.}
Similarly, the local equation of $C_1$ is
$$
y^2 = x^r + \widetilde{a}_{l_0}^{1/2} + \sum_{j=1}^{n} \widetilde{c}_{l_{i_j}} x^{2l_{i_j}}.
$$
Hence $p_a(C_1)=\frac{1}{2}(r-1)=1$, so $r=3$. The local equation of $C_2$ is
\[
w^{4} = z^{3} + \widetilde{a}_{l_0}+ \sum_{j=1}^{n} \widetilde{a}_{l_{i_j}}z^{2l_{i_j}} + \widetilde\psi^2(z).
\]

Combining the two cases, for notational simplicity, we may assume the local equation of $C_2$ is
$$
w^{p^2}=z^{5-p}+\widetilde{b}_{0}+\sum _{i=1}^{n} \widetilde{b}_{l_i}z^{pl_i}+\widetilde\psi^p(z),
$$
where $\widetilde{b}_{0}\in K_2\setminus K_2^p$, $\widetilde{b}_{l_i} \in K_2\setminus K_2^{*p}$, $0<l_1<\cdots<l_n$, $\widetilde\psi(z) \in K_2[z]$, and no monomial of $\widetilde\psi(z)$ belongs to $K_2^p[z^p]$.
\end{proof}

\begin{lemma}
Set $\widetilde{c}_{l_1} \in K_2\setminus K_2^p$, $\widetilde{c}_{l_i} \in K_2\setminus K_2^{*p}$ for $2\leq i\leq n$, $\widetilde\psi(z)\in K_2[z]$ and $0\leq l_1<\cdots<l_n$. Suppose $p\in \{2,3\}$ and the local equation of $C_2$ at $Q_2$ is
$$
w^{p^2}=\widetilde{P}(z):=z^{pl_1+5-p}+\sum _{i=1}^{n} \widetilde{c}_{l_i}z^{pl_i}+\widetilde\psi^p(z).
$$
If $C_2$ is a regular curve over $K_2$, then $l_1=0$.
\end{lemma}

\begin{proof}
Assume $l_1>0$. Set $A := K_2[z,w]/(w^{p^2}-\widetilde{P}(z))$ and $m:= (w^p-\widetilde\psi(0), z)$.  Since $C_2$ is regular at $Q_2=(0,\widetilde\psi(0)^{1/p})$, the local ring
$A_m$ is integrally closed.  Define
$$
\alpha = \frac{w^p-\widetilde\psi(z)}{z^{l_1}}.
$$
Using $w^{p^2}=\widetilde{P}(z)$ one computes
\[
\alpha^p = z^{5-p} + \widetilde{c}_{l_1} + \sum_{i=2}^n \widetilde{c}_{l_i}z^{pl_i-pl_1} \in A,
\]
so $\alpha$ is integral over $A$. Hence $\alpha \in A_m$.
Write $\alpha = u/v$ with $v \notin m$, $u \in A$.
Since $A$ is a free $K_2[z]$-module with basis $1,w,\dots,w^{p^2-1}$,
we may write
\[
v^{p^2} = g_0(z) \in K_2[z], \qquad
u v^{p^2-1} = \sum_{i=0}^{p^2-1} f_i(z) w^i \,\,\, \text{with} \,\,\,f_i(z) \in K_2[z].
\]
Then we have
$$
\left(\sum_{i=0}^{p^2-1} f_i w^i\right)^p = \sum_{j=0}^{p-1} F_j(z)\, w^{jp}
$$
with $F_j(z):=\sum_{i=0}^{p-1}  f _{ip+j}(z)^p \cdot \widetilde{P}(z)^i \in K_2[z]$. As
$$
\frac{\sum_{j=0}^{p-1} F_j(z)\, w^{jp}}{g_0(z)^p}=\left(\frac{uv^{p^2-1}}{v^{p^2}} \right)^p=\alpha^p=z^{5-p} + \widetilde{c}_{l_1} + \sum_{i=2}^n \widetilde{c}_{l_i}z^{pl_i-pl_1},
$$
we obtain
$$
\sum_{j=0}^{p-1} F_j(z)\, w^{jp}=g_0(z)^p \cdot [z^{5-p} + \widetilde{c}_{l_1} + \sum_{i=2}^n \widetilde{c}_{l_i}z^{pl_i-pl_1}].
$$

Thus, we must have
$F_j(z)=0$ for $j\ge 1$, and
$$
F_0(z)=g_0(z)^p \cdot [z^{5-p} + \widetilde{c}_{l_1} + \sum_{i=2}^n \widetilde{c}_{l_i}z^{pl_i-pl_1}].
$$
Combining this with
$$
F_0(0)=\sum_{i=0}^{p-1}f_{ip}(0)^p\cdot \widetilde{P}(0)^i=\sum_{i=0}^{p-1}f_{ip}(0)^p\cdot \widetilde{\psi}(0)^p \in K_2^p,
$$
we have 
$$
F_0(0) = \widetilde{c}_{l_1} \cdot g_0(0)^p \in K_2^p.
$$
If $g_0(0) \neq 0$, then $\widetilde{c}_{l_1} = F_0(0)/g_0(0)^p \in K_2^p$,
contradicting $\widetilde{c}_{l_1} \notin K_2^p$.  Hence $g_0(0)=0$, i.e.
$z \mid g_0(z)$, which forces $g_0(z) \in m$.  This contradicts $v \notin m$. 

Therefore $l_1=0$.
\end{proof}

Adopting the notation of Proposition 3.4, we have the following result.

\begin{lemma}
If $\widetilde{b}_{l_n}\in K_2\setminus K_2^p$ and $pl_n>5$, then the Sylow $p$-subgroup of
$\mathrm{Aut}_{K_2}(C_2)$ is trivial.
\end{lemma}

\begin{proof}
For $p\in\{2,3\}$, set $\lambda = (6-p)/p$; then $\lambda\in\mathbb{N}$ and
$5-p = \lambda p-1$.
Let $\sigma$ be an element of the Sylow $p$-subgroup of $\mathrm{Aut}_{K_2}(C_2)$.

\medskip
\noindent\textbf{Step 1: descent to $C_1$.}
For Equation (4) in Proposition 3.4, we put
\[
x=
\begin{cases}
\dfrac{z}{w^3-\sum_{i=1}^{n}\widetilde{b}_{l_i}^{1/3}z^{l_i}-\widetilde{b}_0^{1/3}-\widetilde\psi(z)}, & p=3,\\[15pt]
\dfrac{w^2-\sum_{i=1}^{n}\widetilde{b}_{l_i}^{1/2}z^{l_i}-\widetilde{b}_0^{1/2}-\widetilde\psi(z)}{z}, & p=2.
\end{cases}
\]
Then $z=x^p$ and
\[
w^p = x^{\lambda p-1} + \widetilde{b}_0^{1/p} + \sum_{i=1}^{n}\widetilde{b}_{l_i}^{1/p}x^{pl_i} + \widetilde\psi(x^p).
\]
Since $\widetilde\psi(x^p)\in K_2[x^p]=K_1^p[x^p]$, we absorb it by setting
$y = w - \widetilde\psi^{1/p}(x^p)$, obtaining the local equation of $C_1$ over $K_1$:
\begin{equation}
y^p = P(x) := x^{\lambda p-1} + b_0 + \sum_{i=1}^{n} b_{l_i} x^{pl_i},
\end{equation}
with $b_0 := \widetilde{b}_0^{1/p}$ and $b_{l_i} := \widetilde{b}_{l_i}^{1/p}$.

\medskip
\noindent\textbf{Step 2: descent to $C_0$.}
Next define
\[
\tau=
\begin{cases}
\dfrac{x}{y - \sum_{i=1}^{n} b_{l_i}^{1/3}x^{l_i} - b_0^{1/3}}, & p=3,\\[14pt]
\dfrac{y - \sum_{i=1}^{n} b_{l_i}^{1/2}x^{l_i} - b_0^{1/2}}{x}, & p=2.
\end{cases}
\]
One verifies that $x=\tau^p$ and
\[
y = \tau^{\lambda p-1} + b_0^{1/p} + \sum_{i=1}^{n} b_{l_i}^{1/p} \tau^{pl_i},
\]
which is the local equation of $C_0\cong \mathbb{P}^1_{K_0}$. An automorphism of $\mathbb{P}^1_{K_0}$ fixing the
unique point $\tau=0$ has the form
$\sigma \tau = \tau/(1+\theta_0 \tau)$ with $\theta_0\in K_0$ (see \cite{s3}, Lemma 3.2).
Therefore,
\[
\sigma x = (\sigma \tau)^p = \frac{x}{1+\theta x}\,\,\, \text{with}\,\,\, \theta:=\theta_0^p\in K_1,
\]

\medskip
\noindent\textbf{Step 3: determining $\theta$.}
In $K_1(C_1)=K_1(x,y)$, we write
\[
\sigma y = \sum_{i=0}^{p-1} f_i(x) y^i\quad \text{with} \,\,\,\,f_i(x)\in K_1(x).
\]
Using $\sigma(y)^p = \sigma(y^p) = P(\sigma x)$ and (5), we have
\begin{equation}
\sum_{i=0}^{p-1} f_i(x)^p P(x)^i=
\frac{x^{\lambda p-1}}{(1+\theta x)^{\lambda p}}
   + \frac{\theta x^{\lambda p}}{(1+\theta x)^{\lambda p}}+b_0
   + \sum_{i=1}^{n} b_{l_i}\Bigl(\frac{x}{1+\theta x}\Bigr)^{pl_i}.
\end{equation}
On the left-hand side, $P(x)^i$ contains $x^{i(\lambda p-1)}$, and since
$f_i(x)^p\in K_1(x^p)$, the term $f_i(x)^p x^{i(\lambda p-1)}$ does not belong to
$K_1(x^p)$ unless $p\mid i$. Comparing terms outside $K_1(x^p)$ forces
$f_i(x)=0$ for $2\le i\le p-1$, while the coefficient of $x^{\lambda p-1}$ yields
$f_1(x)^p = 1/(1+\theta x)^{\lambda p}$, i.e.\ $f_1(x)=1/(1+\theta x)^{\lambda}$.
Substituting these into (6) and multiplying by
$(1+\theta x)^{l_n p}$, we obtain
$$
\begin{aligned}
f_0(x)^p (1+\theta x)^{l_np}=&(1+\theta x)^{l_n p -\lambda p} [\theta x^{\lambda p} -b_0-\sum_{i=1}^n b_{l_i}x^{l_i p}]+b_0(1+\theta x)^{l_n p}\\
&+\sum_{i=1}^{n} b_{l_i} {x}^{l_ip}(1+\theta x)^{l_n p-l_i p}.
\end{aligned}
$$

As $pl_n >5=5-p+p= (\lambda +1)p-1$, we obtain that $l_n\geq \lambda+1$, and that
the highest power of $x$ on the right-hand side of the above equation is $-b_{l_n}\theta^{\,l_n p-\lambda p}x^{(2l_n-\lambda)p}$.
Since $f_0(x)^p(1+\theta x)^{l_np}\in K_1^p(x^p)$, we have
$$
b_{l_n}\theta^{\,l_np-\lambda p}\in K_1^p.
$$
But $b_{l_n}\notin K_1^p$ by hypothesis, hence $\theta=0$.
Consequently $\sigma=\mathrm{id}$, and the Sylow $p$-subgroup of $\mathrm{Aut}_{K_2}(C_2)$ is trivial.
\end{proof}

\begin{rmk}
By Lemma 3.6, bounding the order of the Sylow $p$-subgroup of
$\mathrm{Aut}_{K_2}(C_2)$ reduces to examining only the two remaining
possibilities for the local equation of $C_2$:
\[
\begin{cases}
w^9 = z^2 + \widetilde{b}_0 + \widetilde{b}_1 z^3 + \widetilde\psi^3(z), & p=3,\\[8pt]
w^4 = z^3 + \widetilde{b}_0 + \widetilde{b}_1 z^2 + \widetilde{b}_2 z^4 + \widetilde\psi^2(z), & p=2,
\end{cases}
\]
where in both cases $\widetilde{b}_0\in K_2\setminus K_2^p$,
$\widetilde{b}_i\in K_2\setminus K_2^{*p}$ for $i\ge 1$,
$\widetilde\psi(z)\in K_2[z]$, and no monomial of $\widetilde\psi(z)$ belongs to $K_2^p[z^p]$.
\end{rmk}

\subsubsection{The Sylow $p$-group of $\mathrm{Aut}_{K_2}(C_2)$ when $Q_2$ is semi-rational}
\begin{prop}
If $p=3$ and the local equation of $C_2$ is
\[
w^9 = \widetilde{P}(z) = z^2 + \widetilde{b}_0 + \widetilde{b}_1 z^3 + \widetilde\psi^3(z)
\]
with $\widetilde{b}_0\in K_2\setminus K_2^3$,
$\widetilde{b}_1\in K_2\setminus K_2^{*3}$ and
$\widetilde\psi(z)\in K_2[z]$, then the Sylow $p$-subgroup of
$\mathrm{Aut}_{K_2}(C_2)$ is trivial.
\end{prop}

\begin{proof}
Let $\sigma$ belong to the Sylow $3$-subgroup of $\mathrm{Aut}_{K_2}(C_2)$.
We prove $\sigma=\mathrm{id}$.

\medskip
\noindent\textbf{Step 1: automorphism of $C_1$.}
Set $y = w - \widetilde\psi(z)^{1/3}$ and
\[
x = \frac{z}{w^3 - \widetilde{b}_0^{1/3} - \widetilde{b}_1^{1/3}z - \widetilde\psi(z)} .
\]
Then $z = x^3$ and the local equation of $C_1$ is
\[
y^3 = x^2 + b_0 + b_1 x^3,
\]
where $b_0: = \widetilde{b}_0^{1/3} \in K_1\setminus K_1^3$ and
$b_1: = \widetilde{b}_1^{1/3} \in K_1\setminus K_1^{*3}$.

Introduce $\tau = \dfrac{x}{y - b_0^{1/3} - b_1^{1/3}x}$.
Then $x = \tau^3$ and the local equation of $C_0$ is
$$
y = \tau^2 + b_0^{1/3} + b_1^{1/3}\tau^3.
$$
Every automorphism of $C_0\cong \mathbb{P}^1_{K_0}$ fixing the unique point $\tau=0$ is of the form
$\sigma \tau = \frac{\tau}{1+\theta_0 \tau}$ with $\theta_0\in K_0$.
Hence
\[
\sigma x = \frac{x}{1+\theta x}, \qquad \theta: = \theta_0^3 \in K_1 .
\]

Note that
$$
\sigma \tau=\frac{\tau}{1+\theta_0 \tau}=\frac{x/(y-b_0^{1/3}-b_1^{1/3}x)}{1+\theta^{1/3} x/(y-b_0^{1/3}-b_1^{1/3}x)}=\frac{x}{\theta^{1/3}x+y-b_0^{1/3}-b_1^{1/3}x}
$$
and
$$
\begin{aligned}
\sigma \tau=\sigma \left(\frac{x}{y-b_0^{1/3}-b_1^{1/3}x}\right)=\frac{\sigma x}{\sigma y- b_0^{1/3}-b_1^{1/3}\sigma x}&=\frac{x/(1+\theta x)}{\sigma y- b_0^{1/3}-b_1^{1/3}x/(1+\theta x)}\\
&=\frac{x}{(1+\theta x)(\sigma y-b_0^{1/3})-b_1^{1/3}x},
\end{aligned}
$$
we obtain
$$
\sigma y=\frac{y+(\theta \cdot b_0^{1/3}+\theta^{1/3})x}{1+\theta x}\,\,\,\text{with}\,\,\,\theta \cdot b_0^{1/3}+\theta^{1/3}\in K_1.
$$

\medskip
\noindent\textbf{Step 2: return to $C_2$.}
Since $z = x^3$ and $w = y + \widetilde\psi(z)^{1/3}$,
\[
\sigma z = \frac{z}{1+\theta^3 z}, \qquad
\sigma w = \frac{y + (\theta b_0^{1/3} + \theta^{1/3})x}{1+\theta x} + \sigma\widetilde\psi(z)^{1/3}.
\]
Therefore
\begin{equation}
\begin{aligned}
(\sigma w)^3 &=\frac{y^3+(\theta^3 b_0 + \theta)x^3}{1+\theta^3 x^3}+ \sigma\widetilde\psi(z)\\
&= \frac{w^3}{1+\theta^3 z}
               + \frac{(\theta^3 b_0 + \theta)z - \widetilde\psi(z)}{1+\theta^3 z}
               + \sigma\widetilde\psi(z) .
\end{aligned}
\end{equation}

On the other hand, $\sigma$ acts on $K_2(w,z)/(w^9=\widetilde{P}(z)) = K_2(z)\langle 1,w,\dots,w^8\rangle$,
so we can write $\sigma w = \sum_{i=0}^8 f_i(z) w^i$ with $f_i(z)\in K_2(z)$.
Cubing and using $w^9 = \widetilde{P}(z)$ gives
\begin{equation}
(\sigma w)^3 = \sum_{j=0}^2 F_j(z) w^{3j},
\end{equation}
where $F_j(z) := f_{j}^3 + f_{3+j}^3\widetilde{P}(z) + f_{6+j}^3\widetilde{P}(z)^2$. In particular,
\[
F_1(z) = f_1^3 + f_4^3\widetilde{P}(z) + f_7^3\widetilde{P}(z)^2 .
\]

Comparing the coefficient of $w^3$ in (7) and (8) yields
\begin{equation}
f_1^3+f_4^3 \widetilde{P}(z)+f_7^3 \widetilde{P}(z)^2=\frac{1}{\theta^3 z+1}=\frac{\theta^6 z^2 -\theta^3 z+1}{(\theta^3 z+1)^3}.
\end{equation}
Note that $\widetilde{P}(z) = z^2 + \widetilde{b}_0 + \widetilde{b}_1 z^3 + \widetilde\psi^3(z)$, so
$$
\begin{aligned}
f_1^3+f_4^3 \widetilde{P}(z)+f_7^3 \widetilde{P}(z)^2=&[f_4^3-f_7^3(\widetilde{b}_0 + \widetilde{b}_1 z^3 + \widetilde\psi^3(z))]\cdot z^2+ (zf_7)^3 \cdot z\\
&+[f_7^3(\widetilde{b}_0 + \widetilde{b}_1 z^3 + \widetilde\psi^3(z))^2+f_4^3(\widetilde{b}_0 + \widetilde{b}_1 z^3 + \widetilde\psi^3(z))+f_1^3],
\end{aligned}
$$
comparing the coefficient of $z$ in (9), we have
$$
(zf_7)^3=\frac{-\theta^3}{(\theta^3 z+1)^3},
$$
which implies that $\theta^3=-[z(\theta^3 z+1)f_7]^3 \in K_2^3(z^3)$, so $\theta^3 \in K_2^3$, i.e. $\theta^{1/3} \in K_1$. But $\theta \cdot b_0^{1/3}+\theta^{1/3} \in K_1$ and $b_0\in K_1\setminus K_1^3$, we have $\theta =0$ and $\sigma = \mathrm{id}$.
\end{proof}

\begin{prop}
Let $p=2$ and suppose the local equation of $C_2$ is
\[
w^4 = \widetilde{P}(z) = z^3 + \widetilde{b}_0 + \widetilde{b}_1 z^2 + \widetilde{b}_2 z^4 + \widetilde\psi^2(z),
\]
with $\widetilde{b}_0\in K_2\setminus K_2^2$,
$\widetilde{b}_i\in K_2\setminus K_2^{*2}$ for $i=1,2$, and
$\widetilde\psi(z)\in K_2[z]$ such that no monomial of $\widetilde\psi(z)$ belongs to
$K_2^2[z^2]$.
Then the Sylow $2$-subgroup $H_2$ of $\mathrm{Aut}_{K_2}(C_2)$ satisfies
$|H_2|< 4(p_a(C_2)-1)$.
\end{prop}

\begin{proof}
Let $\sigma\in H_2$.  Set $b_i = \widetilde{b}_i^{1/2}$ ($i=0,1,2$); then
$b_1,b_2\in K_1 \setminus K_1^{*2}$, and $b_0\in K_1\setminus K_1^2$.

\medskip
\noindent\textbf{Step 1: descent to $C_1$ and $C_0$.}
Define
\[
x = \frac{w^2 - b_0 - b_1 z - b_2 z^2 - \widetilde\psi(z)}{z},\qquad
y = w - \widetilde\psi(z)^{1/2}.
\]
One checks that $z=x^2$ and the equation of $C_1$ over $K_1$ is
\[
y^2 = x^3 + b_0 + b_1 x^2 + b_2 x^4.
\]

Now set
\[
\tau = \frac{y - b_0^{1/2} - b_1^{1/2}x - b_2^{1/2}x^2}{x}.
\]
Then $x = \tau^2$ and the equation of $C_0$ over $K_0$ is
\[
y = \tau^3 + b_0^{1/2} + b_1^{1/2}\tau^2 + b_2^{1/2}\tau^4.
\]
Every automorphism of $C_0\cong\mathbb{P}^1_{K_0}$ fixing the unique point $\tau=0$ is of the form
$\sigma \tau = \tau/(1+\theta_0 \tau)$ with $\theta_0\in K_0$.
Hence
\[
\sigma x = (\sigma \tau)^2 = \frac{x}{1+\theta x},\qquad \theta:=\theta_0^2\in K_1,
\]
and
\[
\begin{aligned}
\sigma y& = (\sigma \tau)^3 + b_0^{1/2} + b_1^{1/2}(\sigma \tau)^2 + b_2^{1/2}(\sigma \tau)^4\\
      & =\frac{x}{1+\theta x}\cdot \frac{\tau}{1+\theta^{1/2} \tau}+ b_0^{1/2} +b_1^{1/2}\frac{x}{1+\theta x} + b_2^{1/2}\frac{x^2}{1+\theta^2 x^2}.
\end{aligned}
\]
Note that 
$$
\begin{aligned}
\frac{\tau}{1+\theta^{1/2} \tau}&=\frac{y - b_0^{1/2} - b_1^{1/2}x - b_2^{1/2}x^2}{x+\theta^{1/2}(y - b_0^{1/2} - b_1^{1/2}x - b_2^{1/2}x^2)}\\
&=\frac{x(y - b_0^{1/2} - b_1^{1/2}x - b_2^{1/2}x^2)+\theta^{1/2}(y^2 - b_0 - b_1 x^2 -b_2 x^4)}{x^2+\theta(y^2 - b_0 - b_1 x^2 -b_2 x^4)=x^2+\theta x^3}\\
&=\frac{y - b_0^{1/2} - b_1^{1/2}x - b_2^{1/2}x^2+\theta^{1/2}x^2}{x+\theta x^2},
\end{aligned}
$$
so
$$
\sigma y= \frac{y + (b_0^{1/2}\theta^2 + b_1^{1/2}\theta + \theta^{1/2})x^2}{1+\theta^2 x^2}
$$
with $\delta:=b_0^{1/2}\theta^2 + b_1^{1/2}\theta + \theta^{1/2}\in K_1$.

\medskip
\noindent\textbf{Step 2: return to $C_2$.}
As $z=x^2$ and $w = y + \widetilde\psi(z)^{1/2}$, we obtain
\[
\sigma z = \frac{z}{1+\theta^2 z},\qquad
\sigma w = \frac{w}{1+\theta^2 z} + \frac{\delta z - \widetilde\psi(z)^{1/2}}{1+\theta^2 z} + \sigma\widetilde\psi(z)^{1/2}.
\]
Since $\sigma w\in K_2(w,z)/(w^4-\widetilde{P}(z))=K_2(z)\langle 1,w,w^2,w^3 \rangle$,
\[
\frac{\delta z - \widetilde\psi(z)^{1/2}}{1+\theta^2 z} + \sigma \widetilde\psi(z)^{1/2} \in K_2(z).
\]
Thus,
\[
\delta^2 z^2 - \widetilde\psi(z)+(1+\theta^2 z)^2 \cdot  \sigma \widetilde\psi(z) \in K_2^2(z^2).
\]

Suppose that $\widetilde\psi(z):=\sum_{i=0}^n a_i z^i$ with $n \geq 3$. Then
$$
\begin{aligned}
&\,\,\,\,\,\,\, \delta^2 z^2-\sum_{i=0}^n a_i z^i +(1+\theta^2 z)^2 \cdot \sum_{i=0}^n a_i \left( \frac{z}{1+\theta^2 z} \right)^i \\
&=\frac{[\delta^2 z^2-\sum_{i=0}^n a_i z^i]\cdot (1+\theta^2 z)^{2n-2}+\sum_{i=0}^n a_i z^i(1+\theta^2 z)^{2n-i}}{(1+\theta^2 z)^{2n-2} }\in K_2^2(z^2),
\end{aligned}
$$
hence
\begin{equation}
[\delta^2 z^2-\sum_{i=0}^n a_i z^i]\cdot (1+\theta^2 z)^{2n-2}+\sum_{i=0}^n a_i z^i(1+\theta^2 z)^{2n-i} \in K_2^2[z^2].
\end{equation}

As $n \geq 3$, the highest-degree term of (10) is $-a_n\theta^{4n-4}z^{3n-2}$, which yields 
$$
a_n\theta^{4n-4}z^{3n-2} \in K_2^2[z^2].
$$
But $a_n z^n\notin K_2^2[z^2]$ by the definition of $\widetilde{\psi}(z)$, this forces $\theta=0$.

When $\widetilde\psi(z):=\sum_{i=0}^2 a_i z^i$, then the local equation of $C_2$ is
$$
w^4=z^3 + \widetilde{b}_0 + \widetilde{b}_1 z^2 + \widetilde{b}_2 z^4 + \sum_{i=0}^2 a_i^2 z^{2i}:=z^3+\widetilde{a}_2 z^4+\widetilde{a}_1 z^2+\widetilde{a}_0
$$
with $\widetilde{a}_0:=\widetilde{b}_0+a_0^2 \in K_2\setminus K_2^2$ and $\widetilde{a}_i:=\widetilde{b}_i+a_i^2 \in K_2$ for $1\leq i \leq 2$. Lemma 3.10 and Lemma 3.11 yield $|H_2| < 4(p_a(C_2)-1)$.
\end{proof}

\begin{lemma}
Let $p=2$ and suppose the local equation of $C_2$ is
\[
w^4 = \widetilde{P}(z) = z^3 +\widetilde{a}_2 z^4+\widetilde{a}_1 z^2+\widetilde{a}_0
\]
with $\widetilde{a}_0\in K_2\setminus K_2^2$ and $\widetilde{a}_i \in K_2$ for $1\leq i \leq 2$.
If the Sylow $2$-subgroup $H_2$ of $\mathrm{Aut}_{K_2}(C_2)$ satisfies
$|H_2|\geq 4(p_a(C_2)-1)$, then $C_2$ has infinitely many $K_2$-rational points.
\end{lemma}

\begin{proof}
Let $\pi \colon C_2 \to C_2/H_2$ be the quotient morphism; both curves are regular projective. By the Hurwitz formula (see \cite{s8}, Theorem 7.4.16),
\[
2p_{a}(C_2)-2=\deg(\pi) \cdot (2 p_{a}(C_2/H_2)-2)+\sum_{z\in C_2} d_z \cdot [k(z):K_2],
\]
where $d_z$ denotes the different and $k(z)$ the residue field at $z$. Since $\deg(\pi)=|H_2|\geq 4(p_a(C_2)-1)$ and $p_a(C_2)=\frac{(4-1)(3-1)}{2}=3$, we obtain $p_a(C_2/H_2)=0$. 

If $C_2$ has only finitely many $K_2$-rational points, then infinitely many $c\in K_2^2$ satisfy $\widetilde{P}(c)\in K_2^2\setminus K_2^4$ or $K_2\setminus K_2^2$. In either case, we will derive a contradiction as follows. Hence $C_2$ must possess infinitely many $K_2$-rational points.

\medskip
\noindent\textbf{Case 1:} $\widetilde{P}(c)\in K_2^2\setminus K_2^4$ for infinitely many $c\in K_2^2$.
Choose five distinct elements $c_1,\dots,c_5\in K_2^2$ with
$d_i:=\widetilde{P}(c_i)\in K_2^2$. Thus,
\[
\widetilde{a}_2 c_i^4 + c_i^3 + \widetilde{a}_1 c_i^2 + 0\cdot c_i + \widetilde{a}_0 = d_i.
\]
Then $R\, (\widetilde{a}_2,1,\widetilde{a}_1,0,\widetilde{a}_0)^T = (d_1,\dots,d_5)^T$, where
\[
R = \begin{pmatrix}
c_1^4 & c_1^3 & c_1^2 & c_1 & 1 \\
c_2^4 & c_2^3 & c_2^2 & c_2 & 1 \\
\vdots & \vdots & \vdots & \vdots & \vdots \\
c_5^4 & c_5^3 & c_5^2 & c_5 & 1
\end{pmatrix}
\]
is a Vandermonde matrix. Therefore, we have
\[
0\neq  \det(R) = \prod_{1\le i<j\le 5}(c_i-c_j)^2 \in K_2^2
\]
and $(\widetilde{a}_2,1,\widetilde{a}_1,0,\widetilde{a}_0)^T=R^{-1}\, (d_1,\dots,d_5)^T$. Since all entries of $R$ lie in $K_2^2$ and $d_i \in K_2^2$, we have $\widetilde{a}_0\in K_2^2$, which contradicts $\widetilde{a}_0\notin K_2^2$.

\medskip
\noindent\textbf{Case 2:} $\widetilde{P}(c)\in K_2\setminus K_2^2$ for infinitely many $c\in K_2^2$.
For such a $c$, the maximal ideal at $Q_c=(c,\widetilde{P}(c)^{1/4})$ equals
$(z-c, w^4-\widetilde{P}(c)) = (z-c)$.  Hence the residue field is
$k(Q_c) \cong K_2[w]/(w^4-\widetilde{P}(c))$, a purely inseparable extension
of $K_2$ of degree $4$.  Consequently, every such $Q_c$ is a ramification point
of $\pi$.

Since there are infinitely many such $c$, $\pi$ would have infinitely
many ramification points, which is a contradiction!
\end{proof}

\begin{lemma}
Let $p=2$ and suppose the local equation of $C_2$ is
\[
w^4 = \widetilde{P}(z) = z^3 +\widetilde{a}_2 z^4+\widetilde{a}_1 z^2+\widetilde{a}_0
\]
with $\widetilde{a}_0\in K_2\setminus K_2^2$ and $\widetilde{a}_1,\widetilde{a}_2 \in K_2$.
Then the Sylow $2$-subgroup $H_2$ of $\mathrm{Aut}_{K_2}(C_2)$ satisfies
$|H_2|< 4(p_a(C_2)-1)$.
\end{lemma}

\begin{proof}
Assume $|H_2|\ge 4(p_a(C_2)-1)$.  By Lemma~3.10, $C_2$ possesses infinitely many
$K_2$-rational points.  Choose two distinct such points $(z_1,w_1)$ and
$(z_2,w_2)$.  Then $w_i^4 = z_i^3+\widetilde{a}_2 z_i^4+\widetilde{a}_1 z_i^2+\widetilde{a}_0$, and hence
\[
\Bigl(\frac{w_1}{z_1}\Bigr)^4 - \Bigl(\frac{w_2}{z_2}\Bigr)^4
   = \widetilde{a}_0\Bigl(\frac1{z_1^4}-\frac1{z_2^4}\Bigr)
     + \widetilde{a}_1\Bigl(\frac1{z_1^2}-\frac1{z_2^2}\Bigr)
     + \Bigl(\frac1{z_1}-\frac1{z_2}\Bigr).
\]
Thus,
\[
\rho^4 = \widetilde{a}_0\theta^4 + \widetilde{a}_1\theta^2 + \theta,
\]
where $\theta = 1/z_1-1/z_2$ and $\rho = w_1/z_1-w_2/z_2$.
Now define
\[
\sigma(z) = \frac{z}{1+\theta z},\qquad
\sigma(w) = \frac{w+\rho z}{1+\theta z}.
\]
Using the above relation one verifies directly that
$$
\begin{aligned}
\sigma(w)^4 = \frac{w^4+\rho^4 z^4}{(1+\theta z)^4}
           &= \frac{z^3+\widetilde{a}_2 z^4+\widetilde{a}_1 z^2+\widetilde{a}_0
                  +(\widetilde{a}_0\theta^4+\widetilde{a}_1\theta^2+\theta)z^4}
                  {(1+\theta z)^4}\\
           &=(\frac{z}{1+\theta z})^3+\widetilde{a}_2 (\frac{z}{1+\theta z})^4+\widetilde{a}_1 (\frac{z}{1+\theta z})^2+\widetilde{a}_0= \sigma(\widetilde{P}(z)),
\end{aligned}
$$
so $\sigma $ is an automorphism of $C_2$ of order 2. Since $C_2$ has infinitely many
$K_2$-rational points, choosing different pairs produces infinitely many distinct
automorphisms of $C_2$, contradicting the finiteness of
$\mathrm{Aut}_{K_2}(C_2)$ (which holds because $p_a(C_2)=3\ge2$).
Hence $|H_2| < 4(p_a(C_2)-1)$.
\end{proof}


\subsubsection{The local equation of $C_2$ when $Q_2$ is not semi-rational and $p_a(C_1)=1$}
In this section, we assume that $p=2$, $p_a(C_1)=1$, and $Q_2$ is not semi-rational. 

By Lemma 3.1, there exist positive integers $s$ and $t$ such that $\widetilde{\beta}:=z_0^{2^s} \in K_2 \setminus K_2^2$ and the local equation of $C_2$ at $Q_2$ is
\begin{equation}
w^4=z(z^{2^s}+\widetilde{\beta})^t+\widetilde{\varphi}(z)
\end{equation}
where $\widetilde{\varphi}(z) \in K_2[z^2]$ and no monomial $\widetilde{\varphi}(z)$ of belongs to $K_2^4[z^4] \setminus \{0\}$.

\begin{lemma}
Let $\widetilde{\gamma}_0:=\widetilde\varphi(z_0)$. If $p=2$, $p_a(C_1)=1$, and $Q_2$ is not semi-rational, then $\widetilde\gamma_0 \in K_2$, and
the local equation of $C_2$ at $Q_2$ is
$$
w^4=
\begin{cases}
z(z^2+\widetilde{\beta})+\widetilde{\varphi}(z),& \text{if $\widetilde\gamma_0 \in K_2 \setminus K_2^2$},\\[6pt]
z(z^2+\widetilde{\beta})^t+\widetilde{\varphi}(z), & \text{if $\widetilde\gamma_0 \in K_2^2$},
\end{cases}
$$
where $\widetilde{\varphi}(z) \in K_2[z^2]$ and no monomial $\widetilde{\varphi}(z)$ of belongs to $K_2^4[z^4]\setminus \{0\}$.
\end{lemma}

\begin{proof}
For Equation (11), set 
$$
x=\frac{w^2-\widetilde{\varphi}(z)^{1/2}}{(z^{2^{s-1}}+\widetilde{\beta}^{1/2})^t}.
$$
Then $z=x^2$ and
\begin{equation}
w^2=x(x^{2^s}+\widetilde{\beta}^{1/2})^t+\widetilde{\varphi}(x^2)^{1/2}.
\end{equation}

(i) Suppose $\widetilde\gamma_0=\widetilde{\varphi}(z_0) \notin K_2$. Then $\widetilde\gamma_0^{1/2} \notin K_1$ and the maximal ideal at $Q_1=(z_0^{1/2},\widetilde\gamma_0^{1/4})$ of the curve given by Equation (12) has exactly one generator, namely $x^{2^s}+\widetilde{\beta}^{1/2}$, so Equation (12) is precisely the local equation of $C_1$. As $p_a(C_1)=1$,
$$
\frac{1}{2}(2-1)(2^s t+1-1)=1,
$$
which yields $s=t=1$. Consequently, $z_0^{2}=\widetilde{\beta} \in K_2$. Combining this with $\widetilde{\varphi}(z) \in K_2[z^2]$, we have $\widetilde\gamma_0 \in K_2$, which is a contradiction.
Therefore, $\widetilde\gamma_0 \in K_2$.

(ii) If $\widetilde\gamma_0 \in K_2\setminus K_2^2$, similar to (i), Equation (12) is the local equation of $C_1$. As
$$
1=p_a(C_1)=\frac{1}{2}(2-1)(2^s t+1-1),
$$
we have $s=t=1$.

If $\widetilde\gamma_0 \in K_2^2$, set 
$$
\varphi(x):=\widetilde{\varphi}(x^2)^{1/2},\quad \gamma_0:=\varphi(z_0^{1/2})=\widetilde\gamma_0^{1/2},\quad \beta:=\widetilde{\beta}^{1/2},
$$ 
then $\gamma_0 \in K_1^2$ and $\varphi(x) \in K_1[x^2$].

\medskip
\noindent\textbf{Case 1: $\varphi(x) \equiv \gamma_0$.} We have that $t$ is odd. Otherwise, set
$$
y=\frac{w-\gamma_0^{1/2}}{(x^{2^s}+\beta)^{t/2}},
$$
then Equation (12) becomes $y^2=x$, which is not the local equation of $C_1$ because $C_1$ is non-smooth. Hence, $t$ is odd. Set
$$
y=\frac{w-\gamma_0^{1/2}}{(x^{2^s}+\beta)^{(t-1)/2}},
$$
then
$$
y^2=x(x^{2^s}+\beta)
$$
is exactly the local equation of $C_1$. As $1=p_a(C_1)=\frac{1}{2}(2-1)(2^s +1-1)$, $s=1$.

\medskip
\noindent\textbf{Case 2: $\varphi(x) \not\equiv \gamma_0$.}
There exist $r_0 \in \Bbb{Z}_{>0}$ and $\varphi_1(x)\in K_1[x^2]$ such that $\varphi_1(z_0^{1/2})\neq 0$ and
$$
\varphi(x)-\gamma_0=\varphi_1(x)(x^{2^s}+\beta)^{r_0}.
$$
Thus, Equation (12) becomes
$$
w^2 -\gamma_0= 
\begin{cases}
(x^{2^s}+\beta)^{t}[x+\varphi_1(x)(x^{2^s}+\beta)^{r_0-t}], & \text{if $t \leq r_0$},\\[6pt]
(x^{2^s}+\beta)^{r_0}[x(x^{2^s}+\beta)^{t-r_0}+\varphi_1(x)], & \text{if $t>r_0$},
\end{cases}
$$

\textbf{Case 2-1: $t\le r_0$.} Similar to Case 1, we have that $t$ is odd. Otherwise, set
$$
y=\frac{w-\gamma_0^{1/2}}{(x^{2^s}+\beta)^{t/2}},
$$
then $y^2=x+\varphi_1(x)(x^{2^s}+\beta)^{r_0-t}$ is not the local equation of $C_1$. Thus, $t$ is odd. Set
$$
y=\frac{w-\gamma_0^{1/2}}{(x^{2^s}+\beta)^{(t-1)/2}},
$$
then
$$
y^2=(x^{2^s}+\beta)[x+\varphi_1(x)(x^{2^s}+\beta)^{r_0-t}]
$$
is the local equation of $C_1$. As $1=p_a(C_1)=\frac{1}{2}(2-1)(2^s +1-1)$, $s=1$.

\textbf{Case 2-2: $t> r_0$.} We have that $r_0$ is even. Otherwise, set
$$
y=\frac{w-\gamma_0^{1/2}}{(x^{2^s}+\beta)^{(r_0-1)/2}},
$$
then
$$
y^2=(x^{2^s}+\beta)[x(x^{2^s}+\beta)^{t-r_0}+\varphi_1(x)]
$$
is exactly the local equation of $C_1$. As 
$$
\frac{1}{2}(2-1)[2^s(t-r_0+1) +1-1]\geq t-r_0+1\geq 2,
$$
which contradicts that $p_a(C_1)=1$. Hence, $r_0$ is even. Set
$$
y_1=\frac{w-\gamma_0^{1/2}}{(x^{2^s}+\beta)^{r_0/2}},
$$
then
$$
y_1^2=x(x^{2^s}+\beta)^{t-r_0}+\varphi_1(x).
$$
Similar to (i), we have $\gamma_1:=\varphi_1(z_0^{1/2}) \in K_1$. 

Continuing the case distinction according to whether $\gamma_1 \in K_1 \setminus K_1^2$ or $\gamma_1 \in K_1^2$, and repeating the above steps. Since $\varphi(x)$ has finite order, there exists a sufficiently large integer $n$ such that $\varphi_{n+1}(x)\equiv \varphi_{n+1}(z_0^{1/2})$. Similar to Case 1, we finally obtain $s=1$.
\end{proof}

\begin{rmk}
In fact, $\widetilde\varphi(z)$ in Lemma 3.12 admits a more explicit description.

If $\widetilde{\varphi}(z) \not\equiv \widetilde{\gamma}_0$,
then there exist $r_0 \in \Bbb{Z}_{>0}$ and $\widetilde{\varphi}_1(z)\in K_2[z^2]$ such that 
$\widetilde{\varphi}_1(z_0)\neq 0$ and
$$
\widetilde{\varphi}(z)=\widetilde{\gamma}_0+\widetilde{\varphi}_1(z)(z^2+\widetilde{\beta})^{r_0}.
$$
As $\widetilde{\varphi}_1(z)\in K_2[z^2]$, $\widetilde{\gamma}_1:=\widetilde\varphi_1(z_0) \in K_2$. 
Similarly, if $\widetilde{\varphi}_1(z) \not\equiv \widetilde{\gamma}_1$,
then there exist $r_1 \in \Bbb{Z}_{>0}$ and $\widetilde{\varphi}_2(z)\in K_2[z^2]$ such that 
$\widetilde{\varphi}_2(z_0)\neq 0$ and
$$
\widetilde{\varphi}_1(z)=\widetilde{\gamma}_1+\widetilde{\varphi}_2(z)(z^2+\widetilde{\beta})^{r_1}.
$$

Repeating the above procedure and noting that $\widetilde{\varphi}(z)$ has finite degree, there exists an integer $n$ such that $\widetilde{\varphi}_{n+1}(z)\equiv \widetilde{\varphi}_{n+1}(z_0)$ and
$$
\widetilde{\varphi}_n(z)=\widetilde{\gamma}_n+\widetilde{\varphi}_{n+1}(z)(z^{2}+\widetilde{\beta})^{r_n},
$$
where $\widetilde{\gamma}_n:=\widetilde{\varphi}_{n}(z_0)$ and $r_n \in \Bbb{Z}_{>0}$. Thus,
we have
$$
\widetilde{\varphi}(z)=\widetilde{\gamma}_0+\widetilde{\gamma}_1(z^2+\widetilde{\beta})^{r_0}+\widetilde{\gamma}_2(z^2+\widetilde{\beta})^{r_0+r_1}+\cdots+\widetilde{\gamma}_{n+1}(z^2+\widetilde{\beta})^{r_0+r_1+\cdots+r_n},
$$
where $n\in \Bbb{Z}_{\geq 0}$, $r_i\in \Bbb{Z}_{>0}$, and $\widetilde{\gamma}_i\in K_2$. Select the nonconstant monomials in $\widetilde{\varphi}(z)$ that lie in $K_2^2[z^4+\widetilde{\beta}^2]$, and denote their sum by $\widetilde{\psi}^2(z^2+\widetilde{\beta})$. By definition, $\widetilde{\psi}(z)\in z\cdot K_2[z]$.

Therefore, whether $\widetilde{\varphi}(z) \equiv \widetilde{\gamma}_0$ or not, we may assume
$$
\widetilde{\varphi}(z)=\widetilde{\gamma}_0+\sum_{i=1}^{m} \widetilde{\gamma}_{l_i}(z^2+\widetilde{\beta})^{l_i}+\widetilde{\psi}^2(z^2+\widetilde{\beta}),
$$
where $\widetilde{\gamma}_0\in K_2$, $0<l_1<l_2<\cdots<l_m$,
$\widetilde{\gamma}_{l_i}(z^2+\widetilde{\beta})^{l_i} \in K_2[z^2]\setminus K_2^{*2}[z^4+\widetilde\beta^2]$, and $\widetilde{\psi}(z)\in z\cdot K_2[z]$.
\end{rmk}

By Lemma 3.12 and Remark 3.13, we immediately obtain the following result.

\begin{rmk}
If $\widetilde\gamma_0\in K_2\setminus K_2^2$, then the local equation of $C_2$ at $Q_2$ is
\[
w^4 = z(z^2+\widetilde\beta) + \widetilde\gamma_0 + \sum_{i=1}^{m}\widetilde\gamma_{l_i}(z^2+\widetilde\beta)^{l_i} + \widetilde\psi^2(z^2+\widetilde\beta),
\]
where $\widetilde\beta\in K_2\setminus K_2^2$, $0<l_1<l_2<\cdots<l_m$, $\widetilde\gamma_{l_i}(z^2+\widetilde\beta)^{l_i}\in K_2[z^2]\setminus K_2^{*2}[z^4+\widetilde\beta^2]$, $\widetilde\psi(z)\in z\cdot K_2[z]$, and no monomial of  $\widetilde\psi(z)$ belongs to $K_2^2[z^2]$ (those in $K_2^2[z^2]$ can be absorbed into $w^4$).
\end{rmk}

\begin{thm}
If $p=2$, $p_a(C_1)=1$, and $Q_2$ is not semi-rational, then the local equation of $C_2$ at $Q_2$ is
\[
w^4 = z(z^2+\widetilde\beta) + \widetilde\gamma_0 + \sum_{i=1}^{m}\widetilde\gamma_{l_i}(z^2+\widetilde\beta)^{l_i} + \widetilde\psi^2(z^2+\widetilde\beta),
\]
where $\widetilde\beta\in K_2\setminus K_2^2$, $\widetilde{\gamma}_0 \in K_2$, $0<l_1<l_2<\cdots<l_m$, $\widetilde\gamma_{l_i}(z^2+\widetilde\beta)^{l_i}\in K_2[z^2]\setminus K_2^{*2}[z^4+\widetilde\beta^2]$, $\widetilde\psi(z)\in z\cdot K_2[z]$, and no monomial of  $\widetilde\psi(z)$ belongs to $K_2^2[z^2]$.
\end{thm}

\begin{proof}
By Lemma 3.12 and Remark 3.14, the proof of the theorem reduces to the case 
$$
w^4=z(z^2+\widetilde{\beta})^t+\widetilde{\varphi}(z)
$$
and $\widetilde\gamma_0:=\widetilde{\varphi}(z_0) \in K_2^2$. According to Remark 3.13, we have
$$
\widetilde{\varphi}(z)=\widetilde\gamma_0 + \sum_{i=1}^{m}\widetilde\gamma_{l_i}(z^2+\widetilde\beta)^{l_i} + \widetilde\psi^2(z^2+\widetilde\beta),
$$
where $\widetilde\beta\in K_2\setminus K_2^2$, $0<l_1<l_2<\cdots<l_m$, $\widetilde\gamma_{l_i}(z^2+\widetilde\beta)^{l_i}\in K_2[z^2]\setminus K_2^{*2}[z^4+\widetilde\beta^2]$, $\widetilde\psi(z)\in z \cdot K_2[z]$, and no monomial of  $\widetilde\psi(z)$ belongs to $K_2^2[z^2]$. 

Set $\gamma_0:=\widetilde{\gamma}_0^{\frac{1}{2}}$, $\gamma_{l_i}:=\widetilde{\gamma}_{l_i}^{\frac{1}{2}}$, $\beta:=\widetilde{\beta}^{\frac{1}{2}}$, and $\psi(x):=\widetilde{\psi}^{\frac{1}{2}}(x^2)$.
We will prove $t=1$ by considering cases separately.

\medskip
\noindent\textbf{Case 1: $\widetilde{\varphi}(z)=\widetilde\gamma_0+\widetilde\psi^2(z^2+\widetilde\beta)$.} 
Set
$$
x=\frac{w^2-\widetilde{\varphi}^{\frac{1}{2}}(z)}{(z+\beta)^t},
$$
we have $z=x^2$ and 
$$
w^2=x(x^2+\beta)^t+\widetilde{\varphi}^{\frac{1}{2}}(z)=x(x^2+\beta)^t+\gamma_0+\psi^2(x^2+\beta).
$$

Observe that $t$ is odd. Otherwise, as $\gamma_0 \in K_1^2$, we may set
$$
y=\frac{w-\gamma_0^{\frac{1}{2}}-\psi(x^2+\beta)}{(x^2+\beta)^{t/2}},
$$
then $y^2=x$, which is not the local eqaution of $C_1$. Hence, $t$ is odd. Note that $C_2$ is regular, and that the local equation of $C_2$ at $Q_2$ is
$$
w^4=z(z^{2}+\widetilde{\beta})^t+\widetilde{\gamma}_0+\widetilde{\psi}^2(z^2+\widetilde{\beta}),
$$
so $t=1$ by Lemma 3.16 (i).

\medskip
\noindent\textbf{Case 2: $\widetilde{\varphi}(z)=\widetilde\gamma_0 + \sum_{i=1}^{m}\widetilde\gamma_{l_i}(z^2+\widetilde\beta)^{l_i} + \widetilde\psi^2(z^2+\widetilde\beta)$ with $\widetilde{\gamma}_{l_i}(z^2+\widetilde{\beta})^{l_i} \notin K^2[z^4+\widetilde{\beta}^2]$.} 

In this case, set
$$
x=\frac{w^2-\widetilde{\varphi}^{\frac{1}{2}}(z)}{(z+\beta)^t},
$$
we have $z=x^2$ and 
$$
w^2=x(x^2+\beta)^t+\widetilde{\varphi}^{\frac{1}{2}}(z)=x(x^2+\beta)^t+\gamma_0+ \sum_{i=1}^{m} \gamma_{l_i}(x^2+ \beta)^{l_i}   + \psi^2(x^2+\beta).
$$
Introduce $y_0 = w - \gamma_0^{1/2} - \psi(x^{2}+\beta)$.  Then
\begin{equation}\label{eq:y0}
y_0^2 = x(x^2+\beta)^t + \sum_{i=1}^{m}\gamma_{l_i}(x^2+\beta)^{l_i}.
\end{equation}

\textbf{Case 2-1: $t\le l_1$.}
If $t$ is even, putting $y = y_0/(x^2+\beta)^{t/2}$ would give 
$$
y^2=x+\sum_{i=1}^{m}\gamma_{l_i}(x^2+\beta)^{l_i-t},
$$
which contradicting the non-smoothness of $C_1$. Thus, $t$ is odd. Note that $C_2$ is regular, and that the local equation of $C_2$ is
\[
w^4 = z(z^2+\widetilde\beta)^t + \widetilde\gamma_0 + \sum_{i=1}^{m}\widetilde\gamma_{l_i}(z^2+\widetilde\beta)^{l_i} + \widetilde\psi^2(z^2+\widetilde\beta),
\]
so $t=1$ by Lemma 3.16 (ii).

\textbf{Case 2-2: $t>l_1\ge 1$.}
If $l_1$ is odd, set $y = y_0/(x^2+\beta)^{\frac{l_1-1}{2}}$; then (13) becomes
\[
y^2 = (x^2+\beta)\Bigl[x(x^2+\beta)^{t-l_1} +\gamma_{l_1}+ \sum_{i=2}^{m}\gamma_{l_i}(x^2+\beta)^{l_i-l_1}\Bigr],
\]
which is the local equation of $C_1$, so $p_a(C_1)=\frac{1}{2}(2-1)[2(t-l_1+1)+1-1] = t-l_1+1\ge 2$. This contradicts $p_a(C_1)=1$.  Hence $l_1$ is even. By Lemma 3.16 (iii), this case is imppossible.
\end{proof}

\begin{lemma}
Let $\widetilde\beta \in K\setminus K^2$, $\widetilde{\gamma}_0\in K^2$, $\widetilde{\psi}(z)\in z\cdot K[z]$, 
$0<l_1<\cdots<l_m$, and 
$\widetilde{\gamma}_{l_i}(z^2+\widetilde{\beta})^{l_i} \in K[z]\setminus K^2[z^4+\widetilde{\beta}^2]$ 
for $1\leq i\leq m$.  Suppose $C$ is a projective curve over $K$ of characteristic $p=2$ and its local equation is one of the following:
\begin{enumerate}
\item[(i)] $w^4=\widetilde{P}_1(z):=z(z^{2}+\widetilde\beta)^t+\widetilde{\gamma}_0+\widetilde{\psi}^2(z^2+\widetilde{\beta})$, where $t$ is odd;
\item[(ii)] $w^4=\widetilde{P}_2(z):=z(z^{2}+\widetilde\beta)^t+\widetilde{\gamma}_0+\sum_{i=1}^{m} \widetilde{\gamma}_{l_i}(z^2+\widetilde{\beta})^{l_i}+\widetilde{\psi}^2(z^2+\widetilde{\beta})$, where $t$ is odd and $t\leq l_1$;
\item[(iii)] $w^4=\widetilde{P}_3(z):=z(z^{2}+\widetilde\beta)^t+\widetilde{\gamma}_0+\sum_{i=1}^{m} \widetilde{\gamma}_{l_i}(z^2+\widetilde{\beta})^{l_i}+\widetilde{\psi}^2(z^2+\widetilde{\beta})$, where $t$ is odd, $l_1$ is even, and $t>l_1$.
\end{enumerate}
If $C$ is regular, then its local equation must be of type (i) or (ii), and $t=1$.
\end{lemma}

\begin{proof}
Set $A = K[w,z]/(w^4-\widetilde{P}_j(z))$ ($1\leq j\leq 3$), 
$\mathfrak{M} = (\widetilde\beta^{1/2}, \widetilde\gamma_0^{1/4})$.  Since $C$ is regular, $A_{\mathfrak{M}}$ is integrally closed.
Define
\[
\alpha = \frac{w^2-\widetilde\gamma_0^{1/2}-\widetilde\psi(z^2+\widetilde\beta)}{(z^2+\widetilde\beta)^{s}},
\quad\text{where}\quad
s = \begin{cases}
\frac{t-1}{2}, & \text{in (i) and (ii)},\\[4pt]
\frac{l_1}{2}, & \text{in (iii)}.
\end{cases}
\]
Using $w^4=\widetilde{P}_j(z)$ we obtain
\[
\alpha^2 = \begin{cases}
z(z^2+\widetilde\beta), & \text{case (i)},\\[6pt]
z(z^2+\widetilde\beta) + \sum_{i=1}^{m} \widetilde\gamma_{l_i}(z^2+\widetilde\beta)^{l_i-(t-1)}, & \text{case (ii)},\\[6pt]
z(z^2+\widetilde\beta)^{t-l_1} + \sum_{i=1}^{m} \widetilde\gamma_{l_i}(z^2+\widetilde\beta)^{l_i-l_1}, & \text{case (iii)}.
\end{cases}
\]
In all cases, $\alpha$ is integral over $A$, and hence $\alpha\in A_\mathfrak{M}$.
Write $\alpha = u/v$ with $v\notin \mathfrak{M}$ and $u\in A$. Since $A$ is a free $K[z]$-module with basis $1,w,w^2,w^3$, we may assume
\[
v^4 = g(z)\in K[z], \qquad
uv^3 = f_0(z)+f_1(z)w+f_2(z)w^2+f_3(z)w^3
\]
for some $f_i(z),g(z)\in K[z]$. Then
\[
(uv^3)^2 = g(z)^2\alpha^2.
\]
Substituting the above expressions, we have
$$
[f_0(z)^2+f_2(z)^2 \cdot \widetilde{P}_j(z)]+[f_1(z)^2+f_3(z)^2 \cdot \widetilde{P}_j(z)]\cdot w^2=g(z)^2 \alpha^2.
$$
By comparing coefficients, we obtain
\begin{equation}
f_0(z)^2 + f_2(z)^2\widetilde{P}_j(z) = g(z)^2\alpha^2.
\end{equation}

Now collect the terms involving $z$ on both sides of (14). Thus,
\[
f_2(z)^2 \cdot z(z^2+\widetilde\beta)^t= \begin{cases}
g(z)^2 \cdot z(z^2+\widetilde\beta), & \text{cases (i) and (ii)},\\[6pt]
g(z)^2 \cdot z(z^2+\widetilde\beta)^{t-l_1}, & \text{case (iii)}.
\end{cases}
\]

In (iii), we have $g(z) = f_2(z)(z^2+\widetilde\beta)^{\frac{l_1}{2}}$. As $l_1\ge2$ is even, $g(z)\in \mathfrak{M}$, contradicting $v\notin \mathfrak{M}$. Thus, if $C$ is a regular, its local equation cannot be of type (iii).

In (i) and (ii), we obtain $g(z) = f_2(z)(z^2+\widetilde\beta)^{\frac{t-1}{2}}$. If $t>1$, then $g(z)\in \mathfrak{M}$, contradicting $v\notin \mathfrak{M}$. Therefore, when $C$ is a regular, $t=1$.
\end{proof}

\subsubsection{The Sylow $p$-group of $\mathrm{Aut}_{K_2}(C_2)$ when $Q_2$ is not semi-rational}

\begin{rmk}
Let $\mathrm{PGL}(2,K_0)$ be the automorphism group of $\mathbb{P}_{K_0}^1$,
and $K_0(\tau)$ the function field of $\mathbb{P}_{K_0}^1$.
In characteristic $2$, if $\sigma(\tau) = \frac{a\tau +b}{c\tau +d} \in \mathrm{PGL}(2,K_0)$ has a unique fixed point $\tau = \beta^{1/4}$, then the equation
$$
c\tau^2 + (d-a)\tau - b = 0
$$
has a unique root $\tau = \beta^{1/4}$. Hence $(d-a)^2+4bc=0$, so $d=a$. Thus,
$c \beta^{1/2} =b$. Therefore, we obtain that either $\sigma \tau=\beta^{1/2}/\tau$ or
$$
\sigma \tau=\frac{\tau+\theta_0 \beta^{\frac{1}{2}}}{\theta_0 \tau+1},\,\,\theta_0 \in K_0.
$$

Indeed, in Theorem 3.15, every nontrivial element $\sigma$ of the Sylow $p$-subgroup of $\mathrm{Aut}_{K_2}(C_2)$ has exactly one fixed point $Q_2=(\widetilde{\beta}^{1/2},\widetilde\gamma_0^{1/4})$ on $C_2$ (see \cite{s3}, Lemma 3.2). Viewing $\sigma$ as an element of $\mathrm{PGL}(2,K_0)$, this fixed point corresponds to the point $\tau=\beta^{1/4}$ on $\mathbb{P}^1_{K_0}$, where $\beta:=\widetilde{\beta}^{1/2}$.
\end{rmk}

Adopting the notation of Theorem 3.15, we have the following result.

\begin{thm}
The Sylow $p$-subgroup $H_2$ of $\mathrm{Aut}_{K_2}(C_2)$ has order at most $2$.
\end{thm}

\begin{proof}
\noindent\textbf{Step 1: descent to $C_1$ and $C_0$.}
The local equation of $C_2$ over $K_2$ is
\[
C_2/K_2:\quad w^4 = \widetilde{P}(z) := z(z^2+\widetilde\beta) + \widetilde\gamma_0
                     + \sum_{i=1}^{m}\widetilde\gamma_{l_i}(z^2+\widetilde\beta)^{l_i}
                     + \widetilde\psi^2(z^2+\widetilde\beta).
\]
Set $\beta := \widetilde\beta^{1/2}$, $\gamma_0 := \widetilde\gamma_0^{1/2}$,
$\gamma_{l_i} := \widetilde\gamma_{l_i}^{1/2}$, $\psi(x):=\widetilde\psi(x^2)^{1/2}$,
$$
\delta:=
\begin{cases}
1, & \text{if $\widetilde\gamma_0 \in K_2 \setminus K_2^2$}\\[3pt]
0, & \text{if $\widetilde\gamma_0 \in K_2^2$}
\end{cases}
$$
and define
\[
x = \frac{w^2 - \gamma_0 - \sum_{i=1}^{m}\gamma_{l_i}(z+\beta)^{l_i}
      - \widetilde\psi(z^2+\widetilde\beta)}{z+\beta},\quad
y = w - \psi(x^2+\beta)-(1-\delta) \gamma_0^{1/2}.
\]
Then $z = x^2$, and we obtain the local equation of $C_1$ over $K_1$:
\[
C_1/K_1:\quad y^2 = x(x^2+\beta) + \sum_{i=1}^{m}\gamma_{l_i}(x^2+\beta)^{l_i}+\delta \cdot \gamma_0.
\]
Introduce $h(x) = \sum_{i=1}^{m}\gamma_{l_i}^{1/2}(x+\beta^{1/2})^{l_i}$ and
\[
\tau = \frac{y -\delta \cdot \gamma_0^{1/2} - h(x)}{x + \beta^{1/2}}.
\]
One checks that $x = \tau^2$, and
\[
C_0/K_0:\quad y = \tau(\tau^2+\beta^{1/2}) +\delta \cdot \gamma_0^{1/2} + \sum_{i=1}^{m}\gamma_{l_i}^{1/2}(\tau^2+\beta^{1/2})^{l_i}
\]
is the local equation of $C_0 \cong \mathbb{P}^1_{K_0}$.
By Remark 3.17, to prove that $|H_2|\le 2$, it suffices to consider automorphisms of the following form:
\[
\sigma \tau = \frac{\tau + \theta_0\beta^{1/2}}{\theta_0 \tau + 1},\qquad \theta_0\in K_0.
\]

\medskip
\noindent\textbf{Step 2: automorphism of $C_1$.}
By $x = \tau^2$, we get
\[
\sigma x = (\sigma \tau)^2 = \frac{x + \theta\beta}{\theta x + 1},\quad
\theta := \theta_0^2 \in K_1.
\]
A direct calculation using the relation $\tau = (y-\delta \cdot \gamma_0^{1/2}-h(x))/(x+\beta^{1/2})$
yields
\[
\begin{aligned}
\sigma y &=\sigma \tau \cdot (\sigma x +\beta^{1/2})+\delta \cdot \gamma_0^{1/2} + \sigma h(x)\\
&=\frac{\tau + \theta^{1/2}\beta^{1/2}}{\theta^{1/2} \tau + 1}  \cdot \left( \frac{x + \theta\beta}{\theta x + 1}+\beta^{1/2} \right)+\delta \cdot \gamma_0^{1/2} + \sigma h(x)\\
&=\frac{(\tau + \theta^{1/2}\beta^{1/2})(\theta^{1/2} \tau + 1)}{\theta x + 1}  \cdot \frac{x + \theta\beta+\beta^{1/2}\theta x+\beta^{1/2}}{\theta x + 1}   +\delta \cdot \gamma_0^{1/2} + \sigma h(x)\\
&=\frac{(\beta^{1/2}\theta+1)\tau+\theta^{1/2}(x+\beta^{1/2})}{\theta x + 1}  \cdot \frac{(x + \beta^{1/2})(1+\beta^{1/2}\theta)}{\theta x + 1}   +\delta \cdot \gamma_0^{1/2} + \sigma h(x)\\
&= \frac{[y - \delta \cdot \gamma_0^{1/2} - h(x)](1+\beta\theta^2)
                + (1+\beta^{1/2}\theta)\theta^{1/2}(x^2+\beta)}
                {1+\theta^2 x^2}
           + \delta \cdot \gamma_0^{1/2} + \sigma h(x).
\end{aligned}
\]

\medskip
\noindent\textbf{Step 3: return to $C_2$.}
As $z = x^2$ and $w = y + \psi(x^2+\beta)+(1-\delta) \gamma_0^{1/2}$, we have
$$
\sigma z = \frac{z + \theta^2\beta^2}{\theta^2 z + 1}
$$
and
$$
\begin{aligned}
&\,\,\,\,\,\,\,\,  \sigma w\\
&=\sigma y+\sigma \psi(x^2+\beta)+(1-\delta)\gamma_0^{1/2}\\
&=\frac{[y-\delta \cdot \gamma_0^{\frac{1}{2}}-h(x)](1+\beta \theta^2)+(1+\beta^{\frac{1}{2}} \theta)\theta^{\frac{1}{2}} (x^2+\beta)}{1+\theta^2 x^2}+\gamma_0^{\frac{1}{2}}+\sigma h(x)+\sigma \psi(x^2+\beta)\\
&=\frac{[w-\psi(z+\beta)-\gamma_0^{\frac{1}{2}}-h(x)](1+\beta \theta^2)+(1+\beta^{\frac{1}{2}} \theta)\theta^{\frac{1}{2}} (z+\beta)}{1+\theta^2 z}+\gamma_0^{\frac{1}{2}}+\sigma h(x)\\
&\,\,\,\,\,\,\,+\sigma \psi(z+\beta).
\end{aligned}
$$
Note that $[\sigma h(x)]^2=\sigma \left[\sum_{i=1}^m \gamma_{l_i}(z+\beta)^{l_i} \right]$, so
\begin{equation}
(\sigma w)^2=\frac{w^2(1+\widetilde\beta \theta^4)}{(1+\theta^2 z)^2}+f(z),
\end{equation}
where
$$
\begin{aligned}
f(z):&=\frac{(1+\beta \theta^2)\theta (z^2+\widetilde\beta)-[\psi^2(z+\beta)+\gamma_0+\sum_{i=1}^m \gamma_{l_i}(z+\beta)^{l_i}](1+\widetilde\beta \theta^4)}{(1+\theta^2 z)^2}+\gamma_0\\
&\,\,\,\,\,\,\,+\sigma \left[\sum_{i=1}^m \gamma_{l_i}(z+\beta)^{l_i} \right]+\sigma \psi^2(z+\beta).
\end{aligned}
$$

Since $\sigma$ is an automorphism of $C_2$, we can write
$\sigma w = f_0(z) + f_1(z)w + f_2(z)w^2 + f_3(z)w^3$ with $f_i(z)\in K_2(z)$.
Squaring and using $w^4 = \widetilde{P}(z)$ gives
\begin{equation}
(\sigma w)^2 = [f_0^2 + f_2^2\widetilde{P}(z)] + [f_1^2 + f_3^2\widetilde{P}(z)] w^2.
\end{equation}
From the coefficients of $w^2$ in (15) and (16), we get
\[
\frac{1 + \widetilde\beta\theta^4}{(1+\theta^2 z)^2} = f_1(z)^2 + f_3(z)^2\widetilde{P}(z).
\]
Since $\widetilde{P}(z)$ contains the term $z(z^2+\widetilde\beta)$,
we have $f_3(z)\equiv0$.
Thus
\[
1 + \widetilde\beta\theta^4 = f_1(z)^2 (1+\theta^2 z)^2 \in K_2^2(z^2),
\]
which forces $\widetilde\beta\theta^4\in K_2^2$.  Because $\widetilde\beta\notin K_2^2$
and $\theta^4\in K_2^2$, we must have $\theta=0$.
Therefore $\sigma = \mathrm{id}$. Thus, $|H_2|\leq 2$.
\end{proof}

By Lemma 3.6, Proposition 3.8, Proposition 3.9, and Theorem 3.18, we have the following result.

\begin{thm}
When $p=2$ or 3, if  $p_a(C_1)=1$, then the Sylow $p$-group of $\mathrm{Aut}_{K_2}(C_2)$ has order at most $4(p_a(C_2)-1)$.
\end{thm}

\begin{rmk}
By Remark 2.3, Remark 2.7, and Theorem 3.19, the Sylow $p$-subgroup $H$ in Lemma 2.2 satisfies $|H|\leq 2p(g-1)$. By Tate's genus change formula in \cite{s16}, $(p-1)| 2g$. Thus, we have $|H|\leq 2(g-1)(2g+1)$.
\end{rmk}

\section{Surfaces of general type with $c_2<0$}
Let $k$ be an algebraically closed field of characteristic $p>0$, $S$ a minimal smooth projective surface of general type over $k$. By Lemma 2.2 and Remark 3.20, we have the following results.

\subsection{The group of fibration-preserving automorphisms}
\begin{thm}
Let $f\colon S \rightarrow B$ be a fibration over an algebraically closed field k of characteristic $p > 0$, where S is a minimal smooth projective surface of general type and B is a smooth projective curve of genus $b\geq 2$. Let $C$ be the generic fiber of $f$, and $g:=p_{a}(C)$. Let $K$ be the function field of $B$, and $\overline{C}:=C \times_{K}\overline{K}$. Let $t$ be the number of singular points on $\overline{C}$, and $\pi \colon \widetilde{C} \to \overline{C}$ the normalization. We have
\\(i) if $g(\widetilde{C}) \geq 2$, then $|\mathrm{Aut}_f(S)| \leq \left(\frac{224}{3} \right)^2 g^{4}  b^{4}$.
\\(ii) if $g(\widetilde{C})=1$, then $|\mathrm{Aut}_f(S)|\leq 1792 (g-1) b^{4}$.
\\(iii) if $g(\widetilde{C})=0$ and $t \geq 2$, then $|\mathrm{Aut}_f(S)|\leq \frac{896}{3} g(g-1)(2g+1) b^4$.
\\(iv) if $g(\widetilde{C})=0$ and $t=1$, then $|\mathrm{Aut}_f(S)|\leq \frac{1792}{3}(g-1)(2g+1)^2 b^4$. 
\end{thm}

Combining $K_S^2\geq 2(g-1)(b-1)(1+\frac{1}{g})(1+\frac{g-1}{15g+1})$ (see Lemma 2.7 in \cite{s21}), we have the following corollary.

\begin{cor}
Using the notation of Theorem 4.1, we obtain
\\(i) if $g(\widetilde{C}) \geq 2$, then $|\mathrm{Aut}_f(S)|\leq \left(\frac{6727}{54} \right)^2 (K_{S}^{2})^{4}<15519(K_{S}^{2})^{4}$.
\\(ii) if $g(\widetilde{C})=1$, then $|\mathrm{Aut}_f(S)|\leq 1085\cdot \left(\frac{31}{48} \right)^2 (K_{S}^{2})^{4}<453 (K_{S}^{2})^{4}$.
\\(iii) if $g(\widetilde{C})=0$, then $|\mathrm{Aut}_f(S)|< \frac{175}{3} \left(\frac{31}{12} \right)^4  (K_{S}^{2})^{4}<2598 (K_{S}^{2})^{4}$.
\\ As a conclusion, one has $|\mathrm{Aut}_f(S)|<15519 \cdot (K_{S}^{2})^{4}$.
\end{cor}

\subsection{Automorphisms of surfaces of general type with negative $c_2$}
It is known that there exists a minimal smooth projective surface $S$ of general type with $c_2(S)<0$ in positive characteristic (see \cite{s11}). In \cite{s12}, Shepherd-Barron proved that the Albanese morphism of $S$ induces a surjective morphism $f \colon S \to B$, satisfying the following conditions:
\begin{itemize}
\item $B$ is a smooth projective curve of genus $b \geq 2$, and $f_* \mathcal{O}_S=\mathcal{O}_B$ (hence $f$ is a fibration);
\item The general fiber of $f$ is a rational curve with at least a singular point.
\end{itemize}
By the universal property of the Albanese morphism, for any $\sigma \in \mathrm{Aut}_k(S)$, there exists a morphism $\varphi \colon B \to B$ such that the following diagram commutes
\begin{displaymath}
\xymatrix{
S\ar[r]^{\sigma} \ar[d]_{f} & S \ar[d]^{f}\\
B\ar[r]^{\varphi} & B
}
\end{displaymath}
Since $\sigma$ is an automorphism of $S$ and $\varphi$ is finite, we obtain that $\varphi$ is an automorphism of $B$. Consequently, $|\mathrm{Aut}_k(S)|=|\mathrm{Aut}_f(S)|$. 

With the above preparations, we now present an application of Corollary 4.2.

\begin{thm}
Let $k$ be an algebraically closed field of positive characteristic, and let $S$ be a minimal smooth projective surface of general type over $k$ with $c_2(S)<0$. Then the following assertions hold.

(i) $|\mathrm{Aut}_k(S)|<2598 (K_S^2)^4$.

(ii) If $G$ is an abelian group of $\mathrm{Aut}_k(S)$, then $|G|<81 (K_S^2)^3$.
\end{thm}

\begin{proof}
(i) is immediate from Corollary 4.2 (iii).

(ii) Let $f\colon S\to B$ be the Albanese fibration and $b:=g(B)\geq 2$. It is known that the abelian automorhisms group of $B$ has order at most $4b+4$ (see \cite{s10}). By Lemma 2.2 and Remark 3.20, we have 
$|\mathrm{Aut}_B(S)|\leq 8(g-1)(2g+1)^2$. Thus,

$$
|G|\leq 8(g-1)(2g+1)^2 (4b+4)\leq 31\cdot \left(\frac{155}{96} \right)^2 (K_S^2)^3 <81(K_S^2)^3.
$$
\end{proof}

\begin{example}
Let $k$ be an algebraically closed field of characteristic 2, and let the curve $B$ be defined by
\[
Y^4 Z + YZ^4 = X^5.
\]
The surface $S$, obtained from $B\times \mathbb{P}^1$ by taking the quotient relative to the foliation
\[
D = s^6 \frac{\partial}{\partial s} + \frac{\partial}{\partial x},
\]
is a minimal smooth projective surface of general type with $c_2(S)=-20$ and $K_S^2=32$ (see \cite{s4}, 4.3.3), where $s$ and $x:=X/Z$ are parameters on $\mathbb{P}^1$ and $B$, respectively.

By construction, the Albanese morphism of $S$ induces a fibration $f\colon S\to B$. Over the open subset $B_0:=B\cap\{Z\neq 0\}$ of $B$, $S$ is defined as
\[
Y_0^2 = S_0 T_0^5 + x S_0^6
\]
in the weighted projective space $\operatorname{Proj}(\mathcal{O}_{B_0}[S_0^1,T_0^1,Y_0^3])$, where 
the superscript on each element is its homogeneous degree. Thus, the generic fiber $C$ of $f$ is a rational curve with a unique non-smooth point, whose local equation is
\begin{equation}
Y_0^2 = T_0^5 + x.
\end{equation}

Recall that $\mathrm{Aut}_B(S) \cong \mathrm{Aut}_K(C) \cong H \rtimes M$, where $K$ is the function field of $B$, $H$ is a Sylow $p$-subgroup of $\mathrm{Aut}_K(C)$, and $M$ is a cyclic subgroup of $\mathrm{Aut}_K(C)$. It follows that $\mathrm{Aut}_B(S) \cong \mathbb{Z}_5$:
\begin{itemize}
\item Every element $\sigma \in H$ can be written as
$$
\sigma T_0=T_0/(1+\theta T_0), \quad     \sigma Y_0=f_1(T_0)+f_2(T_0) Y_0,
$$
where $f_1(T_0), f_2(T_0)\in K[T_0]$ and $\theta \in K$. Substituting this into Equation (17) and noting that $x \in K\setminus K^2$, we obtain $H=\{\mathrm{id}\}$. 
\item Every element $\sigma \in M$ can be written as
$$
\sigma T_0=aT_0/(1+b T_0), \quad     \sigma Y_0=h_1(T_0)+h_2(T_0) Y_0,
$$
where $h_1(T_0), h_2(T_0)\in K[T_0]$, $a\in K$ and $b \in K$. Substituting this into Equation (17) yields $M=\{\sigma T_0=\xi T_0 \,\,\text{and}\,\, \sigma Y_0=Y_0 \,|\, \xi^5=1 \} \cong \Bbb{Z}_5$.
\end{itemize}

Therefore,
$$
|\mathrm{Aut}_k(S)|=|\mathrm{Aut}_f(S)|\leq 5\cdot|\mathrm{Aut}_k(B)|=5\times 62400=312000\leq \frac{4875}{16384} (K_S^2)^4.
$$
Morever, let $G$ be an abelian subgroup of $\mathrm{Aut}_k(S)$. Since every abelian subgroup of $\mathrm{Aut}_k(B)$ has order at most $4g(B)+2=26$ (see \cite{s10}), we have
\[
|G|\le 5\cdot (4g(B)+2)=130= \frac{65}{16384} (K_S^2)^3.
\]
\end{example}


\section*{Acknowledgements}
I would like to thank Prof. Wenfei Liu for suggesting this problem.



\begin{thebibliography}{99}

\bibitem{s1}  E. Ballico, \textit{On the automorphisms of surfaces of general type in positive characteristic}, Rend. Mat. Acc. Lincei, \textbf {9(4)} (1993), 121--129.

\bibitem{s2}  J. Cai, \textit{Bounds of automorphisms of surfaces of general type in positive characteristic}, J. Pure Appl. Algebra, \textbf {149} (2000), 241--250.

\bibitem{s3}  X. Faber, \textit{Finite $p$-irregular subgroups of PGL$_2(k)$}, La Matematica, \textbf {2} (2023), 479--522.

\bibitem{s4} Y. Gu, X. Sun, M. Zhou, \textit{Slope inequalities and a Miyaoka-Yau type inequality}, J. Eur. Math. Soc., \textbf {25(2)} (2023), 611--632.


\bibitem{s5} C. D. Hacon, J. M$^C$kernan , C. Xu, \textit{On the birational automorphisms of varieties of general type}, Annals of mathematics,  \textbf {177}, (2013), 1077--1111.

\bibitem{s6} A. T. Huckleberry, M. Sauer, \textit{On the order of automorphism group of a surface of general type}, Math. Z., \textbf {205} (1990), 321--329.

\bibitem{s7} A. Hurwitz, $\ddot{U}$\textit{ber algebraische Gebilde mit eindeutigen Transformationen in sich}, Mathematische Annalen, \textbf {41(3)} (1893), 403--442.

\bibitem{s8} Q. Liu, \textit{Algebraic Geometry and Arithmetic Curves}, Oxford University Press. Vol. \textbf {6} (2002).

\bibitem{s9} X. Lv, S. Tan, \textit{Sharp bound on abelian automorphism groups of surfaces of general type}, Memoirs of the American Mathematical Society, \textbf{317(1532)} (2025).


\bibitem{s10} S. Nakajima, \textit{On abelian automorphism groups of algebraic curves}, J. London Math. Soc., \textbf {s2-36(1)} (1987), 23--32.

\bibitem{s11} M. Raynaud, \textit{Contre-example au ``vanishing theorem" en caract$\mathrm{\acute{e}}$ristique $p>0$. In : C. P. Ramanujam---a Tribute}, Tata Inst. Fund. Res. Stud. Math., \textbf{8}, Springer, Berlin, (1978), 273--278.

\bibitem{s12} N. I. Shepherd-Barron. \textit{Geography for surfaces of general type in positive characteristic}, Invent. Math., \textbf {106} (1991), 263--274.  

\bibitem{s13} I. Shimada, \textit{On supercuspidal families of curves on a surface in positive characteristic}, Math. Ann., \textbf {292} (1992), 645--669.

\bibitem{s14} H. Stichtenoth, $\ddot{U}$\textit{ber die Automorphismengruppe eines algebraischen Funktionenk}$\ddot{o}$\textit{rpers von Primzahlcharakteristik. Teil \uppercase \expandafter{\romannumeral 1}: Eine Absch}$\ddot{a}$\textit{tzung der Ordnung der Automorphismengruppe}, Arch. Math., (Basel), \textbf {24} (1973), 527--544. 

\bibitem{s15} E. Szab\'{o}, \textit{Bounding automorphism groups}, Math. Ann., \textbf{304} (1996), 801--811.

\bibitem{s16} J. Tate, \textit{Genus change in inseparable extensions of function fields}, Proc. Amer. Math. Soc., \textbf {3} (1952), 400--406.

\bibitem{s17} G. Xiao, \textit{Bound of automorphisms of surfaces of general type, \uppercase \expandafter{\romannumeral 1}}, Annals of Mathematics. \textbf {139} (1994), 51--77.

\bibitem{s18} G. Xiao, \textit{Bound of automorphisms of surfaces of general type, \uppercase \expandafter{\romannumeral 2}}. J. Algebraic Geom., \textbf {4} (1995), 701--793.

\bibitem{s19} G. Xiao, \textit{On abelian automorphism groups of surfaces of general type}, Invent. Math., \textbf {102} (1990), 619--631.

\bibitem{s20} S. Yang, X. Yu, Z. Zhu, \textit{On automorphism groups of smooth hypersurfaces}, Journal of Algebraic Geometry, \textbf {34} (2025), 579--611.

\bibitem{s21} X. Zhong, \textit{Bound of automorphisms of fibred surfaces in positive characteristic}, J. Algebra., \textbf {667} (2025),725--745.



\end{thebibliography}
\end{document}